\documentclass{amsart}

\usepackage{mathrsfs}
\usepackage{amscd}
\usepackage{amsmath}
\usepackage{amssymb}
\usepackage{amsthm}
\usepackage{epsf}
\usepackage{latexsym}
\usepackage{verbatim}
\usepackage[all, cmtip]{xy}
\usepackage{tikz}
\usetikzlibrary{positioning}
\usetikzlibrary{matrix}
\usepackage{float}
\usepackage{hyperref}
\usepackage{comment}
\usepackage{enumitem}
\usepackage{mathtools}
\usepackage{xcolor}

\tikzstyle{bsq}=[rectangle, draw, thick, minimum width=.5cm, minimum height=.5cm]
\tikzstyle{bver}=[rectangle, draw, thick, minimum width=1cm, minimum height=2cm]
\tikzstyle{bhor}=[rectangle, draw, thick, minimum width=2cm, minimum height=1cm]

\usepackage[left=3.2cm, right=3.2cm]{geometry}

\newtheorem{theorem}{Theorem}[section]
\newtheorem{lemma}[theorem]{Lemma}

\newtheorem{corollary}[theorem]{Corollary}
\newtheorem{proposition}[theorem]{Proposition}

\newtheorem{varexample}[theorem]{Example}

\theoremstyle{definition}

\def\Z{{\mathbb Z}}

\def\Q{{\mathbb Q}}

\newcommand{\cA}{\mathcal{A}}
\newcommand{\cB}{\mathcal{B}}

\newcommand{\Jac}{\operatorname{Jac}}
\newcommand{\val}{\operatorname{val}}
\newcommand{\Pic}{\operatorname{Pic}}

\newcommand{\gon}{\operatorname{gon}}

\begin{document}
\title{Fibonacci Sumsets and the Gonality of Strip Graphs}
\author[D. Jensen]{David Jensen}
%\address{David Jensen: Department of Mathematics,  University of Kentucky \hfill \newline\texttt{}
%\indent 733 Patterson Office Tower,
%Lexington, KY 40506--0027, USA}
\email{{\tt dave.jensen@uky.edu}}
\author[D. Rivera Laboy]{Doel Rivera Laboy}
%\address{David Jensen: Department of Mathematics,  University of Kentucky \hfill \newline\texttt{}
%\indent 733 Patterson Office Tower,
%Lexington, KY 40506--0027, USA}
\email{{\tt doel.riveralaboy@uky.edu}}
\date{}
\bibliographystyle{alpha}

\begin{abstract}
We provide a new perspective on the divisor theory of graphs, using additive combinatorics.  As a test case for this perspective, we compute the gonality of certain families of outerplanar graphs, specifically the strip graphs.  The Jacobians of such graphs are always cyclic of Fibonacci order.  As a consequence, we obtain several results on the additive properties of Fibonacci numbers.
\end{abstract}

\maketitle

\section{Introduction}
\label{Sec:Intro}

\subsection{Divisors and Gonality}
\label{Sec:Divisors}
The divisor theory of graphs was first introduced in \cite{BakerNorine07, Baker08}.  This theory, which mirrors and informs that of divisors on algebraic curves, has attracted a great deal of interest in the intervening 17 years.  In \cite{Baker08}, Baker defines a new graph invariant, the \emph{gonality}, which is the minimum degree of a divisor of positive rank.  More generally, once can define the $r$-\emph{gonality}, the minimum degree of a divisor of rank $r$.  Computing the gonality of a graph is NP-hard \cite{GSvdW}, and there is a wealth of literature on bounding this invariant \cite{deBruynGijswijt, HJJS,vDdBSvdW, DEM21}.

In this paper, we recast the gonality problem in the language of additive combinatorics.  The group of equivalence classes of degree-0 divisors on a graph $G$ is a finite abelian group, known variously as the \emph{Jacobian} $\Jac(G)$, the \emph{sandpile group}, or the \emph{critical group}.  This group contains a subset $\cA(G) \subset \Jac(G)$, whose elements are in bijection with equivalence classes of vertices in $G$.  Additive combinatorics is largely concerned with the additive structure of subsets of abelian groups, known as \emph{additive sets}.  Of particular interest are the \emph{iterated sumsets} $m\cA(G)$, consisting of all sums of $m$ elements of $\cA(G)$.  In Corollary~\ref{Cor:Gonality}, we provide an equivalent characterization of the $r$-gonality in terms of the iterated sumsets of $\cA(G)$.

\subsection{Outerplanar Graphs}
\label{Sec:IntroOuterplanar}

As a test case for this perspective, we consider the gonality of \emph{outerplanar graphs}, which are graphs that can be embedded in the plane so that all vertices belong to a single face.  More specifically, we focus on two families of outerplanar graphs, the \emph{fan graphs} $\mathcal{F}_n$ (see Figure~\ref{Fig:Fan}) and the \emph{strip graphs} $G_n$ (see Figure~\ref{Fig:Strip}).

These families of graphs form a good test case for two reasons.  First, the lower bounds on gonality that can be found in the existing literature are insufficient for computing the gonality of these graphs.  For example, it is shown in \cite{deBruynGijswijt} that the gonality is bounded below by a much-studied graph invariant known as the \emph{treewidth}.  In Lemma~\ref{Lem:Treewidth}, however, we show that the treewidth of an outerplanar graph is at most two, which is as small as possible.  Similarly, in \cite{HJJS}, it is shown that the gonality is bounded below by the so-called \emph{scramble number}, but in Lemma~\ref{Lem:Scramble}, we show that the scramble number of the strip graph $G_n$ is three.

Second, the Jacobian of an outerplanar graph is relatively easy to understand.  More specifically, if $G$ is a maximal outerplanar graph with exactly two vertices of valence two, then $\Jac(G)$ is cyclic of Fibonacci order (see Corollary~\ref{Cor:Iso}).  Throughout, we let $F_n$ denote the $n$th Fibonacci number, indexed so that $F_0 = 0$ and $F_1 = 1$.  For the specific families of graphs $\mathcal{F}_n$ and $G_n$, the additive sets $\cA(G) \subset \Jac(G)$ admit nice descriptions in terms of Fibonacci numbers.  Specifically, the set $\cA(\mathcal{F}_n)$ consists of 0 and all odd-index Fibonacci numbers between 1 and $F_{2n}$.  In other words,
\[
\cA(\mathcal{F}_n) = \{ 0 \} \cup \{ F_{2k-1} \mid 1 \leq k \leq n \} \subset \Z/F_{2n}\Z .
\]
The set $\cA(G_n)$ also admits a nice description:
\[
\cA(G_n) = \{ F_k F_{2n-k} \mid 0 \leq k \leq n \} \subset \Z/F_{2n}\Z .
\]
By Catalan's identity, $\cA(G_n)$ is a translate of the subset
\[
\cB(G_n) = \{ (-1)^{k+1} F_k^2 \mid 0 \leq k \leq n \} \subset \Z/F_{2n}\Z .
\]

The gonality of the fan graph $\mathcal{F}_n$ was computed in \cite{Hendrey18}.  The back half of this paper is primarily focused on computing the gonality of the strip graphs $G_n$.

\begin{theorem}
\label{Thm:MainThm}
We have
\[
\gon(G_n) = \begin{cases}
        \lceil \frac{n+1}{2} \rceil & \text{ if $n \leq 7$} \\
        5 & \text{ if $n \geq 8$}.
    \end{cases}
\]
\end{theorem}

\subsection{Fibonacci Numbers}
\label{Sec:Fibonacci}

While our focus in this paper is to use additive combinatorics to compute graph invariants like the gonality, one can also go the other way, using graph invariants to discover additive properties of Fibonacci numbers.  We briefly mention a few results, which are purely statements about Fibonacci numbers, that follow from our graph-theoretic investigation.  The first follows directly from \cite[Theorem~1.7]{BakerNorine07}.

\begin{theorem}
\label{Thm:GenusFib}
Every integer $\pmod{F_{2n}}$ can be expressed as a sum of $n-1$ elements of $\cA(\mathcal{F}_n)$, $\cA(G_n)$, or $\cB(G_n)$.  In other words,
\begin{align*}
(n-1)\cA(\mathcal{F}_n) &= \Z/F_{2n}\Z \\
(n-1)\cA(G_n) &= \Z/F_{2n}\Z \\
(n-1)\cB(G_n) &= \Z/F_{2n}\Z.
\end{align*}
Moreover, there exists an integer $\pmod{F_{2n}}$ that cannot be expressed as a sum of $n-2$ elements of $\cA(\mathcal{F}_n)$, $\cA(G_n)$, or $\cB(G_n)$.  In other words,
\begin{align*}
(n-2)\cA(\mathcal{F}_n) & \subsetneq \Z/F_{2n}\Z \\
(n-2)\cA(G_n) & \subsetneq \Z/F_{2n}\Z \\
(n-2)\cB(G_n) & \subsetneq \Z/F_{2n}\Z.
\end{align*}
\end{theorem}

The next two results are derived from the fact that automorphisms of the graph $G$ induce Freiman isomorphisms of the additive set $\cA(G) \subset \Jac(G)$.

\begin{theorem}
\label{Thm:FreimanFan}
Let $1 \leq k_1 , \ldots, k_m , k'_1, \ldots, k'_m \leq n$.  Then
\[
\sum_{i=1}^m F_{2k_i-1} = \sum_{i=1}^m F_{2k'_i-1} \pmod{F_{2n}} \iff \sum_{i=1}^m F_{2n-2k_i+1} = \sum_{i=1}^m F_{2n-2k'_i+1} \pmod{F_{2n}}.
\]
\end{theorem}

\begin{theorem}
\label{Thm:FreimanStrip}
Let $0 \leq k_1, \ldots, k_m, k'_1, \ldots, k'_m \leq n$.  Then
\begin{align*}
\sum_{i=1}^m F_{k_i}F_{2n-k_i} &= \sum_{i=1}^m F_{k'_i}F_{2n-k'_i} \pmod{F_{2n}} \iff \\
\sum_{i=1}^m F_{n-k_i}F_{n+k_i} &= \sum_{i=1}^m F_{n-k'_i}F_{n+k'_i} \pmod{F_{2n}}.
\end{align*}
\end{theorem}

Finally, the gonalities of these graphs can be reinterpreted in the following way.

\begin{theorem}
\label{Thm:FanGonality}
There exists an integer $x \in \Z/F_{2n}\Z$ such that:
\begin{enumerate}
\item  $x$ can written as a sum of $d-1$ or fewer odd-index Fibonacci numbers between 1 and $F_{2n-1}$ and
\item  for all $k \leq n$, $x$ can be written as a sum of $d$ or fewer such numbers, with $F_{2k-1}$ as one of the summands,
\end{enumerate}
if and only if
\[
d \geq \phi_n \coloneq \min \Big\{\lfloor \sqrt{n+1} \rfloor - 1 + \Big\lceil \frac{n+1-\lfloor \sqrt{n+1} \rfloor}{\lfloor \sqrt{n+1} \rfloor} \Big\rceil , \lceil \sqrt{n+1} \rceil - 1 + \Big\lceil \frac{n+1-\lceil \sqrt{n+1} \rceil}{\lceil \sqrt{n} \rceil} \Big\rceil \Big\} .
\]
\end{theorem}

\begin{theorem}
\label{Thm:StripGonality}
There exists an integer $x \in \Z/F_{2n}\Z$ such that, for all $j \leq n$, $x$ can be written as a sum of $d$ integers of the form $F_k F_{2n-k}$, with $F_j F_{2n-j}$ as one of the summands, if and only if $d \geq \min \{ \lceil \frac{n+1}{2} \rceil, 5 \}$.
\end{theorem}

\subsection{Outline of the Paper}

In Section~\ref{Sec:Preliminaries}, we introduce the basic theory of divisors on graphs, including the monodromy pairing, and its relation to counts of certain types of spanning forests.  In Section~\ref{Sec:Outerplanar}, we begin our discussion of outerplanar graphs, proving in Corollary~\ref{Cor:Iso} that the Jacobian of certain outerplanar graphs is always cyclic of order $F_{2n}$.  In Section~\ref{Sec:Fan}, we discuss the fan graphs, describing the set $\cA(\mathcal{F}_n)$ and proving Theorems~\ref{Thm:FreimanFan} and~\ref{Thm:FanGonality}.  Then, in Section~\ref{Sec:Strip}, we turn to the strip graphs, proving Theorems~\ref{Thm:GenusFib} and Theorem~\ref{Thm:FreimanStrip}.  We also prove, in Theorems~\ref{Thm:BoundOf3} and~\ref{Thm:BoundOf4}, that the gonality of $G_n$ is bounded below by 4 for $n \geq 6$.  Finally, in Section~\ref{Sec:MainThm}, we improve this bound to 5 for $n \geq 8$, proving Theorem~\ref{Thm:MainThm}.  This last section involves a significant amount of tedious computation, but the strategy of proof is the same as for the simpler lower bounds in Theorems~\ref{Thm:BoundOf3} and~\ref{Thm:BoundOf4}.  Readers who are uninterested in these technical calculations are encouraged to read the proofs of these simpler theorems instead.

\subsection*{Acknowledgements}

This research was supported by NSF DMS-2054135.

\section{Preliminaries}
\label{Sec:Preliminaries}

\subsection{Divisors on Graphs}
\label{Sec:Divisors}

In this section, we describe the basic theory of divisors on graphs.  For more detail, we refer the reader to \cite{Baker08, BakerJensen16}.

Let $G$ be a graph with $n+1$ vertices, no loops, and possibly with parallel edges.  Throughout, we fix an ordering $v_0 , \ldots , v_n$ of the vertices of $G$.  The vertex $v_0$ will be referred to as the \emph{base vertex}.  A \emph{divisor} $D$ on $G$ is a formal linear combination of vertices $D = \sum_{i=0}^n a_i \cdot v_i$.  We may think of a divisor as an integer vector of length $n+1$.

The \emph{graph Laplacian} of $G$ is the $(n+1) \times (n+1)$ matrix with rows and columns indexed by the vertices of $G$, and whose $(i,j)$th entry is
\begin{displaymath}
\Delta_{i,j} = \left\{ \begin{array}{ll}
\val(v_i)&\text{if }i=j\\
-\text{\# of edges between }v_i\text{ and }v_j&\text{if }i\neq j.
\end{array} \right.
\end{displaymath}
That is, $\Delta$ is the difference of the valency matrix and the adjacency matrix.

Considering $\Delta$ as a map from $\Z^{n+1}$ to $\Z^{n+1}$, its image is the set of \emph{principal divisors}.  Two divisors $D$ and $D'$ are \emph{equivalent} if their difference is principal.  In other words, $D$ is equivalent to $D'$ if there exists $\vec{v} \in \Z^{n+1}$ such that $D-D' = \Delta \vec{v}$.  The set of equivalence classes of divisors on $G$ is a group under addition, known as the \emph{Picard group}
\[
\Pic (G) = \Z^{n+1} /\Delta \Z^{n+1} .
\]

The \emph{degree} of a divisor $D = \sum_{i=0}^n a_i \cdot v_i$ is the integer $D = \sum_{i=0}^n a_i$.  Since all principal divisors have degree zero, the degree of a divisor is invariant under equivalence.  The group of equivalence classes of divisors of degree zero is called the \emph{Jacobian} $\Jac(G)$.  It is a consequence of Kirchoff's matrix tree theorem that $\vert \Jac(G) \vert$ is equal to the number of spanning trees in $G$, often denoted $\kappa (G)$.

\subsection{The Monodromy Pairing}
\label{Sec:Pairing}

Since every row of the graph Laplacian $\Delta$ sums to zero, $\Delta$ is not invertible.  Every matrix $\Delta$, however, has a \emph{generalized inverse} -- that is, a matrix $L$ such that $\Delta L \Delta = \Delta$.

One way to construct a generalized inverse is as follows.  Let $\widetilde{\Delta}$ be the $n \times n$ matrix obtained from $\Delta$ by deleting the first row and first column.  Now, let $L$ be the $(n+1) \times (n+1)$ matrix obtained by appending a zero row and zero column to the top and left of $\widetilde{\Delta}^{-1}$.  Then $L$ is a generalized inverse of $\Delta$.  For the remainder of the paper, we fix $L$ to be this particular generalized inverse.

In \cite{Shokrieh}, Shokrieh defines the \emph{monodromy pairing} on $\Jac (G)$.  The pairing
\[
\langle \cdot , \cdot , \rangle \colon \Jac(G) \times \Jac(G) \to \Q/\Z
\]
is given by:
\[
\langle D, D' \rangle = [D]^T L [D'] \pmod{\Z} .
\]
The monodromy pairing is independent of the choice of generalized inverse $L$.

For specific generators of $\Jac(G)$, the monodromy pairing can be explicity described as follows.  Let $\kappa_{i,j}(G)$ denote the number of 2-component spanning forests of $G$ such that one component contains the base vertex $v_0$ and the other component contains both $v_i$ and $v_j$.

\begin{theorem}
\label{Thm:Minors}
We have
\[
\langle v_i - v_0, v_j -v_0 \rangle = L_{i,j} = \frac{\kappa_{i,j}(G)}{\kappa (G)} .
\]
\end{theorem}

\begin{proof}
The fact that $\langle v_i - v_0, v_j -v_0 \rangle = L_{i,j}$ is immediate from the definition of the monodromy pairing.  If either $i$ or $j$ is equal to 0, then $L_{i,j} = \kappa_{i,j}(G) = 0$.  By Cramer's rule, for $i,j \neq 0$, we have
\[
L_{i,j} = \frac{C_{i,j}}{\mathrm{det}(\widetilde{\Delta})},
\]
where $C_{i,j}$ is the $(i,j)$th cofactor of $\widetilde{\Delta}$.  By the matrix tree theorem, we have
\[
\mathrm{det} (\widetilde{\Delta}) = \vert \Jac(G) \vert = \kappa (G),
\]
and by the all minors matrix tree theorem from \cite{Chaiken}, we have $C_{i,j} = \kappa_{i,j}(G)$.
\end{proof}

\subsection{The Rank of a Divisor}
\label{Sec:Rank}

A divisor $D = \sum_{i=0}^n a_i \cdot v_i$ is called \emph{effective} if $a_i \geq 0$ for all $i$.  If a divisor $D$ is not equivalent to an effective divisor, we say that it has rank $-1$.  Otherwise, we define the \emph{rank} of a divisor $D$ to be the maximum integer $r$ such that $D-E$ is equivalent to an effective divisor for all effective divisors $E$ of degree $r$.

In this section, we reinterpret the rank of a divisor in the language of additive combinatorics.  Define the set
\[
\cA(G) \coloneq \{ v_i - v_0 \mid 0 \leq i \leq n \} \subseteq \Jac(G) .
\]
For a positive integer $m$, we write
\[
m\cA(G) \coloneq \{ a_1 + \cdots + a_m \mid a_i \in \cA(G) \} .
\]
Since $0 \in \cA(G)$, we have $(m-1)\cA(G) \subseteq m\cA(G)$ for all $m$.  Note that a divisor $D$ of degree $d$ is equivalent to an effective divisor if and only if $D-dv_0 \in \cA(G)$.  This yields the following result.

\begin{lemma}\cite[Theorem~1.7]{BakerNorine07}
\label{Lem:Genus}
For any graph $G$, the first Betti number $g = \vert E(G) \vert - \vert V(G) \vert + 1$ is the smallest integer such that $g\cA(G) = \Jac(G)$.
\end{lemma}

Given a divisor $D$, we write
\[
D - \cA(G) = \{ D-a \mid a \in \cA(G) \} .
\]

\begin{proposition}
\label{Prop:Rank}
Let $D$ be a divisor of degree $d$ on a graph $G$.  Then $D$ has rank at least $r$ if and only if
\[
(D-dv_0) - r\cA(G) \subseteq (d-r)\cA(G) .
\]
\end{proposition}

\begin{proof}
By definition, $D$ has rank at least $r$ if, for all effective divisors $E$ of degree $r$, $D-E$ is equivalent to an effective divisor -- that is, if $D-E-(d-r)v_0 \in (d-r)\cA(G)$.  Since $E$ is effective if and only if $E-rv_0 \in r\cA(G)$, we see that $D$ has rank at least $r$ if and only if $(D-dv_0) - r\cA(G) \subseteq (d-r)\cA(G)$.
\end{proof}

The $r$-\emph{gonality} $\gon_r(G)$ of a graph $G$ is the minimum degree of a divisor with rank at least $r$.  The 1-gonality is typically just called the gonality.  By Proposition~\ref{Prop:Rank}, we have the following.

\begin{corollary}
\label{Cor:Gonality}
For a graph $G$, we have
\[
\gon_r(G) = \min \{ d \mid \exists D \in \Jac(G) \text{ such that } D-r\cA(G) \subseteq (d-r)\cA(G) \} .
\]
\end{corollary}

Note that, since $0 \in r\cA(G)$, if $D-r\cA(G) \subseteq (d-r)\cA(G)$, then $D \in (d-r)\cA(G)$.  We will make frequent use of this simple observation in the proofs of Theorems~\ref{Thm:MainThm}, ~\ref{Thm:BoundOf3}, and~\ref{Thm:BoundOf4}.  We also have the following lower bound.

\begin{corollary}
\label{Cor:Bound}
Let $H$ be an abelian group and let $\varphi \colon \Jac(G) \to H$ be a homomorphism.  Then
\[
\gon_r(G) \geq \min \{ d \mid \exists D \in \Jac(G) \text{ such that } \varphi(D) - r\varphi(\cA(G)) \subseteq (d-r)\varphi(\cA(G)) \} .
\]
\end{corollary}

Corollary~\ref{Cor:Bound} is useful in applications because, for any vertex $v_i$ in $G$, we have the homomorphism
\[
\langle \cdot , v_i-v_0 \rangle \colon \Jac(G) \to \Q/\Z ,
\]
and we can use Theorem~\ref{Thm:Minors} to completely describe the set
\[
\langle \cA(G) , v_i-v_0 \rangle = \{ \kappa_{i,j}(G) \mid 0 \leq j \leq n \} \subseteq \Z/\kappa(G)\Z \subset \Q/\Z .
\]
Since every finite subgroup of $\Q/\Z$ is cyclic, this allows us to bound the gonality of a graph using techniques from additive combinatorics on cyclic groups.

\subsection{Freiman Isomorphisms}
\label{Sec:Freiman}

A fundamental concept in additive combinatorics is that of the \emph{Freiman isomorphism}.  Let $\cA , \cB$ be subsets of abelian groups $G$ and $H$, respectively.  A \emph{Freiman isomorphism} of order $m$ from $\cA$ to $\cB$ is a bijection $\psi \colon \cA \to \cB$ such that
\[
\sum_{i=1}^m a_i = \sum_{i=1}^m a'_i \iff \sum_{i=1}^m \psi (a_i) = \sum_{i=1}^m \psi (a'_i)
\]
for all $a_1 , \ldots , a_m , a'_1 , \ldots , a'_m \in \cA$.  The following simple observation will be useful for identifying Freiman isomorphisms.

\begin{proposition}
\label{Prop:Freiman}
Let $\psi$ be an automorphism of a graph $G$.  Then the induced map $\psi_* \colon \cA(G) \to \cA(G)$ given by $\psi_* (v_i - v_0) = \psi(v_i) - v_0$ is a Freiman isomorphism of arbitrary order.
\end{proposition}

\begin{proof}
We have
\[
\sum_{i=1}^m (v_{j_i} - v_0) \sim \sum_{i=1}^m (v_{j'_i} - v_0) \iff \sum_{i=1}^m v_{j_i} \sim \sum_{i=1}^m v_{j'_i}.
\]
Similarly,
\[
\sum_{i=1}^m (\psi(v_{j_i}) - v_0) \sim \sum_{i=1}^m (\psi(v_{j'_i}) - v_0) \iff \sum_{i=1}^m \psi(v_{j_i}) \sim \sum_{i=1}^m \psi(v_{j'_i}).
\]
Because $\psi$ is an automorphism, we have
\[
\sum_{i=1}^m v_{j_i} \sim \sum_{i=1}^m v_{j'_i} \iff \sum_{i=1}^m \psi(v_{j_i}) \sim \sum_{i=1}^m \psi(v_{j'_i})
\]
and the result follows.
\end{proof}

\section{Outerplanar Graphs}
\label{Sec:Outerplanar}

\subsection{Outerplanar Graphs}
We now turn our attention to a specific family of graphs.  A graph $G$ is called \emph{outerplanar} if it can be embedded in the plane in such a way that all vertices belong to a single face.  An outerplanar graph is called \emph{maximal} if it is simple, and adding an edge between any two non-adjacent vertices results in a non-outerplanar graph.  Equivalently, a simple outerplanar graph on $n+1$ vertices is maximal if and only if it has $2n-1$ edges.  Note that the first Betti number of a maximal outerplanar graph is $n-1$.

Two examples of maximal outerplanar graphs that we will discuss in this paper are the \emph{fan graphs}, pictured in Figure~\ref{Fig:Fan} and the \emph{strip graphs}, pictured in Figure~\ref{Fig:Strip}.  The \emph{fan graph} $\mathcal{F}_n$ is the graph with $n+1$ vertices, and edges between $v_0$ and $v_i$ for all $i \geq 1$, and between $v_i$ and $v_j$ if $\vert i-j \vert = 1$ for all $i,j \geq 1$.  See Figure~\ref{Fig:Fan}.

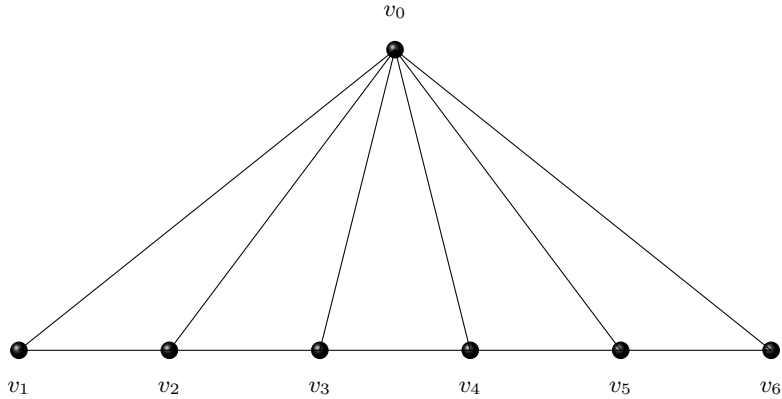
\begin{figure}[h]
\begin{tikzpicture}[scale=2.0]

\draw [ball color=black] (0,0) circle (0.55mm);
\draw [ball color=black] (1,0) circle (0.55mm);
\draw [ball color=black] (2,0) circle (0.55mm);
\draw [ball color=black] (3,0) circle (0.55mm);
\draw [ball color=black] (4,0) circle (0.55mm);
\draw [ball color=black] (5,0) circle (0.55mm);
\draw [ball color=black] (2.5,2) circle (0.55mm);
\draw (0,0)--(1,0);
\draw (1,0)--(2,0);
\draw (2,0)--(3,0);
\draw (3,0)--(4,0);
\draw (4,0)--(5,0);
\draw (2.5,2)--(0,0);
\draw (2.5,2)--(1,0);
\draw (2.5,2)--(2,0);
\draw (2.5,2)--(3,0);
\draw (2.5,2)--(4,0);
\draw (2.5,2)--(5,0);
\draw (2.5,2.25) node {{\small $v_0$}};
\draw (0,-0.25) node {{\small $v_1$}};
\draw (1,-0.25) node {{\small $v_2$}};
\draw (2,-0.25) node {{\small $v_3$}};
\draw (3,-0.25) node {{\small $v_4$}};
\draw (4,-0.25) node {{\small $v_5$}};
\draw (5,-0.25) node {{\small $v_6$}};

\end{tikzpicture}
\caption{The fan graph $\mathcal{F}_6$.}
\label{Fig:Fan}
\end{figure}

The \emph{strip graph} $G_n$ is the graph with $n+1$ vertices and edges between $v_i$ and $v_j$ if $\vert i-j \vert \in \{1,2\}$.  See Figure~\ref{Fig:Strip}.

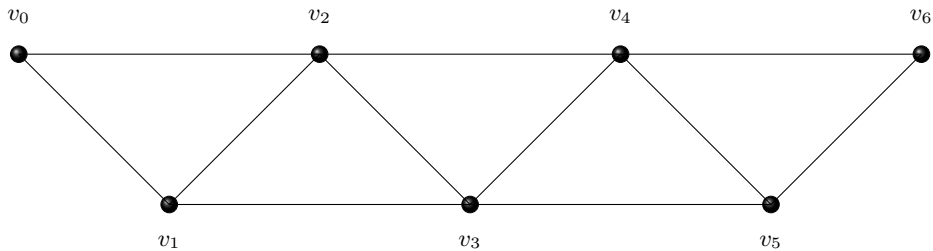
\begin{figure}[h]
\begin{tikzpicture}[scale=2.0]

\draw [ball color=black] (0,1) circle (0.55mm);
\draw [ball color=black] (1,0) circle (0.55mm);
\draw [ball color=black] (2,1) circle (0.55mm);
\draw [ball color=black] (3,0) circle (0.55mm);
\draw [ball color=black] (4,1) circle (0.55mm);
\draw [ball color=black] (5,0) circle (0.55mm);
\draw [ball color=black] (6,1) circle (0.55mm);
\draw (0,1)--(1,0);
\draw (2,1)--(1,0);
\draw (2,1)--(3,0);
\draw (4,1)--(3,0);
\draw (4,1)--(5,0);
\draw (6,1)--(5,0);
\draw (0,1)--(6,1);
\draw (1,0)--(5,0);
\draw (0,1.25) node {{\small $v_0$}};
\draw (1,-0.25) node {{\small $v_1$}};
\draw (2,1.25) node {{\small $v_2$}};
\draw (3,-0.25) node {{\small $v_3$}};
\draw (4,1.25) node {{\small $v_4$}};
\draw (5,-0.25) node {{\small $v_5$}};
\draw (6,1.25) node {{\small $v_6$}};

\end{tikzpicture}
\caption{The strip graph $G_6$.}
\label{Fig:Strip}
\end{figure}

\subsection{The Jacobian of an Outerplanar Graph}
\label{Sec:OuterplanarJacobian}

Throughout, we let $F_n$ denote the $n$th Fibonacci number, indexed so that $F_0 = 0$ and $F_1 = 1$.  In \cite{Slater77}, Slater shows that, if $G$ is a maximal outerplanar graph with $n+1$ vertices, exactly 2 of which have valence 2, then $\kappa (G) = F_{2n}$.  In addition, we will show that the Jacobian of any such graph is cyclic.  We first need the following preliminary lemma.

\begin{lemma}
\label{Lem:OuterplanarKappa2}
Let $G$ be a maximal outerplanar graph with $n+1$ vertices, exactly 2 of which have valence 2.  Let $v_0$ be a vertex of valence 2 and $v_1$ a vertex adjacent to $v_0$.  Then $\kappa_{1,1} (G) = F_{2n-1}$.
\end{lemma}

\begin{proof}
The invariant $\kappa_{1,1} (G)$ counts the number of 2-component spanning forests such that $v_0$ is in one component and $v_1$ is in the other.  Such forests are in bijection with spanning trees containing the edge from $v_0$ to $v_1$.  Specifically, given a spanning tree containing this edge, remove it to obtain a 2-component spanning forest such that $v_0$ is in one component and $v_1$ is in the other.  This operation is clearly invertible -- given a 2-component spanning forest such that $v_0$ is in one component and $v_1$ is in the other, add the edge from $v_0$ to $v_1$ to obtain a spanning tree.

Now, let $G'$ be the graph obtained by deleting the two edges adjacent to $v_0$.  As shown in the proof of \cite[Proposition~1]{Slater77}, $G'$ is a maximal outerplanar graph on $n$ vertices, exactly 2 of which have valence 2. By \cite[Proposition~1]{Slater77}, we have $\kappa (G') = F_{2n-2}$.  It follows that the number of spanning trees in $G$ that do not contain the edge from $v_0$ to $v_1$ is $F_{2n-2}$, hence the number of spanning trees in $G$ that do contain this edge is $F_{2n} - F_{2n-2} = F_{2n-1}$.
\end{proof}

\begin{corollary}
\label{Cor:Iso}
Let $G$ be a maximal outerplanar graph with $n+1$ vertices, exactly 2 of which have valence 2.  Let $v_0$ be a vertex of valence 2 and $v_1$ a vertex adjacent to $v_0$.
The map $\varphi \colon \Jac(G) \to \Z/F_{2n}\Z \subset \Q/\Z$ given by
\[
\varphi (D) = \langle D, v_1-v_0 \rangle \pmod{\Z}
\]
is an isomorphism.  In particular, $\Jac(G) \cong \Z/F_{2n}\Z$.
\end{corollary}

\begin{proof}
Combining Theorem~\ref{Thm:Minors} with Lemma~\ref{Lem:OuterplanarKappa2}, we have
\[
\varphi (v_1-v_0) = \frac{\kappa_{1,1}(G)}{\kappa(G)} = \frac{F_{2n-1}}{F_{2n}} \in \Q/\Z .
\]
Since $F_{2n-1}$ and $F_{2n}$ are relatively prime, this element generates the cyclic subgroup of order $F_{2n}$ in $\Q/\Z$, hence $\varphi$ maps $\Jac(G)$ onto this cyclic subgroup.  Since, by \cite[Proposition~1]{Slater77}, we have $\vert \Jac(G) \vert = \kappa(G) = F_{2n}$, it follows that $\varphi$ is an isomorphism onto this subgroup.
\end{proof}

\subsection{Bounds on the Gonality of Outerplanar Graphs}
\label{Sec:OuterplanarBounds}

In the next sections, we will discuss the gonality of certain families of outerplanar graphs.  Here, we note that bounds in the existing literature are insufficient for computing the gonality of these graphs.  In \cite{deBruynGijswijt}, it is shown that a well-known graph invariant, the \emph{treewidth}, is a lower bound on gonality.  Outerplanar graphs, however, have treewidth at most 2, and as we shall see in the later sections, the gonality is often much higher.

\begin{lemma}
\label{Lem:Treewidth}
Let $G$ be an outerplanar graph.  Then $\mathrm{tw} (G) \leq 2$.
\end{lemma}

\begin{proof}
Both outerplanar graphs and graphs of treewidth at most 2 have forbidden minor characterizations.  Specifically, a graph is outerplanar if and only if it has neither of the forbidden minors $K_4$ nor $K_{2,3}$ \cite[Exercise~4.23]{Diestel18}.  Similarly, a graph has treewidth at most 2 if and only if it does not have the forbidden minor $K_4$ \cite{Bodlaender98}.  It follows that the treewidth of an outerplanar graph is at most 2.
\end{proof}

A divisor $D = \sum_{i=0}^n a_i \cdot v_i$ on a graph is \emph{multiplicity-free} if $a_i$ is equal to either 0 or 1 for all $i$.  In \cite{DEM21}, Dean, Everett, and Morrison define the \emph{multiplicity-free gonality} $\mathrm{mfgon}(G)$ of a graph $G$ to be the minimum degree of a multiplicity-free divisor of positive rank.  Of course, the multiplicity-free gonality is an upper bound on the gonality.  For maximal outerplanar graphs, however, the gonality is typically much smaller than the multiplicity-free gonality.  Recall that an \emph{independent set} in a graph $G$ is a set of vertices, no two of which are adjacent.  The \emph{independence number} $\alpha (G)$ is the maximal size of an independent set.

\begin{lemma}
\label{Lem:MFGon}
Let $G$ be a simple planar graph on $n+1$ vertices, with all of its faces triangles, except for possibly the outer face.  Then $\mathrm{mfgon} (G) = n+1 - \alpha (G)$.   In particular, if $G$ is maximal outerplanar, then $\mathrm{mfgon} (G) \geq \frac{n}{2}$.
\end{lemma}

\begin{proof}
By \cite[Lemma~2.4]{DEM21}, if $S \subset V(G)$ is an independent set, then
\[
D = \sum_{v_i \notin S} v_i
\]
is a multiplicity-free divisor of positive rank.  It follows that $\mathrm{mfgon} (G) \leq n+1 - \alpha (G)$.

Now, let $D$ be a multiplicity-free divisor of degree less than $n+1-\alpha(G)$.  Then there exists a pair of adjacent vertices $v$ and $w$ so that neither $v$ nor $w$ is in the support of $D$.  Now, run Dhar's burning algorithm starting at $v$.  By assumption, any pair of adjacent vertices is contained in a triangle, and since $D$ has at most 1 chip on each vertex, if two vertices of the triangle burn, then so does the third.  It follows by induction that every vertex burns, hence $D$ is $v$-reduced.  Since $v$ is not in the support of $D$, it follows that $D$ does not have positive rank.

For the final statement, note that a maximal outerplanar graph contains a Hamiltonian cycle, so the independence number $\alpha(G)$ is less than or equal to that of the cycle, which is $\lceil \frac{n}{2} \rceil$.
\end{proof}

\section{Fan Graphs}
\label{Sec:Fan}

The set $\cA(\mathcal{F}_n)$ has a particularly nice description.

\begin{lemma}
\label{Lem:Fibonacci}
For all $1 \leq k \leq n$, we have $\kappa_{1,k}(\mathcal{F}_n) = F_{2n-2k+1}$.
\end{lemma}

\begin{proof}
By Lemma~\ref{Lem:OuterplanarKappa2}, we have $\kappa_{1,1} (\mathcal{F}_n) = F_{2n-1}$.  (Note that, in Lemma~\ref{Lem:OuterplanarKappa2}, it is $v_0$ rather than $v_1$ that has valence 2.  However, the number $\kappa_{1,1}$, which counts the number of 2-component spanning forests such that $v_0$ is in one component and $v_1$ is in the other, is invariant under switching the labels of $v_0$ and $v_1$.)

Now assume that $k \geq 2$.  We will show that $\kappa_{1,k} (\mathcal{F}_n) = \kappa_{1,1} (\mathcal{F}_{n-k+1})$.  From the previous paragraph, it then follows that $\kappa_{1,k} (\mathcal{F}_n) = F_{2n-2k+1}$.  Given a 2-component spanning forest, let $T_0$ denote the component containing $v_0$ and $T_1$ denote the component containing $v_1$ and $v_k$.  Then $T_1$ must contain the unique path from $v_1$ to $v_k$ that does not pass through $v_0$.  By deleting this path, we obtain a 2-component spanning forest in $\mathcal{F}_n \smallsetminus \{v_1, \ldots , v_{k-1} \} \cong \mathcal{F}_{n-k+1}$ such that one component contains $v_0$, and the other contains $v_k$.  This operation is clearly invertible, hence this yields a bijection between the two sets of 2-component spanning forests, and $\kappa_{1,k} (\mathcal{F}_n) = \kappa_{1,1} (\mathcal{F}_{n-k+1})$.
\end{proof}

Since, by Corollary~\ref{Cor:Iso}, the map $\varphi \colon \Jac(\mathcal{F}_n) \to \Z/F_{2n}\Z$ is an isomorphism, we may identify $\Jac(\mathcal{F}_n)$ with its image under $\varphi$.  By Lemma~\ref{Lem:FibProduct}, under this identification, we have the following.

\begin{corollary}
\label{Cor:FanSetA}
The set $\cA(\mathcal{F}_n)$ consists of 0 and all odd-index Fibonacci numbers between 1 and $F_{2n}$.  In other words,
\[
\cA(\mathcal{F}_n) = \{ 0 \} \cup \{ F_{2k-1} \mid 1 \leq k \leq n \} \subset \Z/F_{2n}\Z .
\]
\end{corollary}

Theorems~\ref{Thm:FreimanFan} and~\ref{Thm:FanGonality} follow immediately.

\begin{proof}[Proof of Theorem~\ref{Thm:FreimanFan}]
There is an involution of $\mathcal{F}_n$ that fixes $v_0$ and sends $v_k$ to $v_{n+1-k}$ for all $k \geq 1$.  Thus, by Proposition~\ref{Prop:Freiman} the involution $\iota \colon \cA(\mathcal{F}_n) \to \cA(\mathcal{F}_n)$ given by $\iota (0) = 0$ and $\iota (F_{2k-1}) = F_{2n-2k+1}$
is a Freiman isomorphism of arbitrary order.
\end{proof}

\begin{proof}[Proof of Theorem~\ref{Thm:FanGonality}]
By \cite[Theorem~11]{Hendrey18}, $\gon (\mathcal{F}_n) = \phi_n$, and by Corollary~\ref{Cor:Gonality}, we have
\[
\gon(\mathcal{F}_n) = \phi_n = \min \{ d \mid \exists x \in \Z/F_{2n}\Z \text{ such that } x-\cA(\mathcal{F}_n) \subseteq (d-1)\cA(\mathcal{F}_n) \} .
\]
Finally, by Corollary~\ref{Cor:FanSetA}, we have $x - \cA(\mathcal{F}_n) \subseteq (d-1)\cA(\mathcal{F}_n)$ if and only if $x$ satisfies the two conditions in the statement of the theorem.
\end{proof}

\section{Strip Graphs}
\label{Sec:Strip}

\subsection{The Set $\cA(G_n)$}
\label{Sec:SetA}
We now turn to the strip graphs $G_n$.  We aim to represent the set $\cA(G_n)$ explicitly as a subset of $\Z/F_{2n}\Z$.  To do this, we will use Theorem~\ref{Thm:Minors}.

\begin{lemma}
\label{Lem:Kappa2}
We have
\begin{align*}
\kappa_{1,k}(G_n) &= F_{2n-2} + \kappa_{1,k-2}(G_{n-2}) \mbox{ for } 2 \leq k \leq n \\
\kappa_{1,1}(G_n) &= F_{2n-1} \\
\kappa_{1,0} (G_n) &= 0.
\end{align*}
\end{lemma}

\begin{proof}
To see that $\kappa_{1,0} (G_n)=0$, note that if one component of a 2-component forest contains $v_0$, then the other does not.  The fact that $\kappa_{1,1} (G_n) = F_{2n-1}$ is Lemma~\ref{Lem:OuterplanarKappa2}.

We now assume that $k \geq 2$.  Given a 2-component spanning forest, let $T_0$ denote the component containing $v_0$ and $T_1$ denote the component containing $v_1$ and $v_k$.  Note that the edge from $v_0$ to $v_1$ cannot appear in such a spanning forest.  There are two cases:

\begin{enumerate}
\item  If the edge from $v_0$ to $v_2$ is not in $T_0$, then $T_0 = \{v_0\}$. As such, $T_1$ is a spanning tree of the graph $G_n \smallsetminus \{v_0\} \cong G_{n-1}$. By \cite[Proposition~1]{Slater77} (or \cite[Lemma~1]{CombiSpanning}), there are exactly $F_{2n-2}$ such spanning trees.

\item If the edge from $v_0$ to $v_2$ is in $T_0$, then the edge from $v_1$ to $v_3$ is in any path from $v_1$ to $v_k$. As such, this edge is contained in $T_1$.  Thus, the restriction of our spanning forest to $G_n \smallsetminus \{ v_0, v_1 \} \cong G_{n-2}$ is a 2-component spanning forest, where one component contains $v_2$ and the other component contains both $v_3$ and $v_k$.  By definition, the number of such spanning forests is equal to $\kappa_{1,k-2}(G_{n-2})$.
\end{enumerate}
Combining the two cases, we obtain
\[
\kappa_{1,k}(G_n) = F_{2n-2} + \kappa_{1,k-2}(G_{n-2}).
\]
\end{proof}

As in Section~\ref{Sec:Fan}, since the map $\varphi$ is an isomorphism, we may identify the set $\cA(G_n)$ with its image under $\varphi$.

\begin{lemma}
\label{Lem:FibProduct}
For all $0 \leq k \leq n$, we have $\kappa_{1,k}(G_n) = F_k F_{2n-k}$.
\end{lemma}

\begin{proof}
We fix $n$ and prove this by induction on $k$.  Note that the base cases $k=0,1$ are done in Lemma~\ref{Lem:Kappa2}.  For $k \geq 1$, by Lemma~\ref{Lem:Kappa2}, we have
\begin{align*}
    \kappa_{1,k+1}(G_n) &= F_{2n-2} + \kappa_{1,k-1} (G_{n-2}) \\
    &= F_{2n-2} + F_{k-1}F_{2n-k+1},
\end{align*}
where the second equality holds by induction.  Now, by the identity on the top of page 48 from \cite{CountingHosoya}, the above is equal to
\begin{align*}
    &\mathcolor{white}{=} F_{k-1}F_{2n-k-2} + F_k F_{2n-k-1} + F_{k-1}F_{2n-k-3}\\
    &= F_k F_{2n-k-2} + F_{k-1}F_{2n-k-2} + F_k F_{2n-k-3} + F_{k-1} F_{2n-k-3} \\
    &= F_{k+1}F_{2n-k-2} + F_{k+1}F_{2n-k-3} \\
    &= F_{k+1}F_{2n-k-1}.
\end{align*}
\end{proof}

\begin{corollary}
\label{Cor:SetA}
We have
\[
\cA(G_n) = \{ F_k F_{2n-k} \mid 0 \leq k \leq n \} \subset \Z/F_{2n}\Z .
\]
\end{corollary}

Theorems~\ref{Thm:GenusFib} and~\ref{Thm:FreimanStrip} follow immediately.

\begin{proof}[Proof of Theorem~\ref{Thm:GenusFib}]
The statements about $\cA(\mathcal{F}_n)$ and $\cA(G_n)$ follow directly from Lemma~\ref{Lem:Genus}, using the fact that the first Betti number of $\mathcal{F}_n$ is $n-1$.  To see the statement about $\cB(G_n)$, note that by the Catalan identity, one has $F_{n-k} F_{n+k} - F_n^2 = (-1)^{k+1} F_k^2$.  Since translation by $-F_n^2$ is a Freiman isomorphism of arbitrary order, the result follows.
\end{proof}

\begin{proof}[Proof of Theorem~\ref{Thm:FreimanStrip}]
There is an involution of $G_n$ that sends $v_k$ to $v_{n-k}$ for all $k$.  Thus, by Proposition~\ref{Prop:Freiman}, the involution $\iota \colon \cA(G_n) \to \cA(G_n)$ given by $\iota (F_k F_{2n-k}) = F_{n-k}F_{n+k}$
is a Freiman isomorphism of arbitrary order.
\end{proof}

\subsection{The Zeckendorf Form}
\label{Sec:Zeckendorf}
Our goal for the remainder of the paper is to use Corollary~\ref{Cor:Gonality} to compute the gonality of $G_n$.  To do this, we need to describe the sets $m\cA(G_n)$ for certain small values of $m$.  In this section, we introduce a fundamental tool for describing these sets.

Every nonnegative integer can be written uniquely as a sum of non-consecutive Fibonacci numbers.  This expression is called the \emph{Zeckendorf form} of the number.  When $x \in \Z/F_{2n}\Z$, we define the Zeckendorf form of $x$ to be the Zeckendorf form of its unique representative in the range $0 \leq x < F_{2n}$.  We will primarily be interested in the \emph{leading terms} of a number written in Zeckendorf form, which are the largest Fibonacci numbers appearing in this sum.  Equivalently, the leading term of a number $x$ is the largest Fibonacci number smaller than $x$.  Our next goal is to write every element of $\cA(G_n)$ in Zeckendorf form.

\begin{lemma}
\label{Lem:ZeckendorfProduct}
\cite[Theorem~1]{Freitag1998}
For $m\geq n$, the Zeckendorf form of $F_mF_n$ is:
\[
F_m F_n = \begin{cases}
        \sum_{r=1}^{\lfloor \frac{n}{2}\rfloor} F_{m+n+2-4r} & \text{ if $n$ is even}\\
        F_{m-n+1} + \sum_{r=1}^{\lfloor \frac{n}{2}\rfloor} F_{m+n+2-4r} & \text{ if $n$ is odd}.\\
    \end{cases}.
\]
\end{lemma}

Because of this, we have the following corollary.

\begin{corollary}
\label{Cor:A}
All elements of $\cA(G_n)$ have the following Zeckendorf form:
\[
F_k F_{2n-k} = \begin{cases}
        \sum_{r=1}^{\lfloor \frac{k}{2}\rfloor} F_{2n+2-4r} & \text{ if $k$ is even}\\
        F_{2n-2k+1} + \left(\sum_{r=1}^{\lfloor \frac{k}{2}\rfloor} F_{2n+2-4r} \right)& \text{ if $k$ is odd}.\\
    \end{cases}
\]
As such, all non-zero elements of $\cA(G_n)$, except $F_{2n-1}$, have $F_{2n-2}$ as the leading term of their Zeckendorf form.
\end{corollary}

Corollary~\ref{Cor:A} can also be seen by induction, using Lemma~\ref{Lem:Kappa2}.

\subsection{A Lower Bound on the Gonality of Strip Graphs}
\label{Sec:StripBounds}

In Section~\ref{Sec:MainThm}, we will prove that $\gon (G_n) = 5$ when $n$ is sufficiently large.  To show that $\gon (G_n) \leq 5$, it suffices to exhibit a divisor of degree 5 and positive rank.  The more difficult part of the argument is to establish a lower bound on the gonality.  Here, we demonstrate a lower bound of 3 when $n \geq 4$.  Later in the paper, we will improve this bound to 4  when $n \geq 6 $ and then to 5 when $n \geq 8$.  These arguments all follow a similar approach, but the latter bounds are much more involved, with many more cases.  We structure the argument in this way is to highlight the technique in a simpler case before proving the main theorem.

The \emph{scramble number} of a graph was first defined in \cite{HJJS}, where it was shown that it is a lower bound on gonality.  Here, we show that the scramble number of $G_n$ is 3 for all $n \geq 4$.

\begin{lemma}
\label{Lem:Scramble}
If $n \geq 4$, then $\mathrm{sn} (G_n) = 3$.
\end{lemma}

\begin{proof}
The graph $G_n$ has a topological subgraph isomorphic to the graph $C_{3;2,2,1}$ from \cite{EagletonMorrison22}, where it is shown to have scramble number 3.  Since the scramble number is topological subgraph monotone, we see that $\mathrm{sn}(G_n) \geq 3$.  

To show that $\mathrm{sn}(G_n) \leq 3$, we use \cite[Theorem~1]{CFGMMMORW}, which shows that the scramble number is bounded above by a graph invariant known as the \emph{screewidth}.  For $i$ an even number less than $n$, let $X_i = \{v_i , v_{i+1} \}$, and if $n$ is even, let $X_n = \{ v_n \}$.  Let $\mathcal{T}$ be the path with nodes $X_i$ where $X_i$ is adjacent to $X_{i+2}$ for all $i$.  The adhesion of each link in $\mathcal{T}$ is either 2 or 3, and the adhesion of each node is either 1 or 2.  Thus, the width of this tree-cut decomposition is 3, hence $\mathrm{scw}(G_n) \leq 3$.
\end{proof}

Lemma~\ref{Lem:Scramble} implies that, for $n \geq 4$, the gonality of $G_n$ is at least 3.  We can also prove this using our approach.  We find it helpful to illustrate our approach first in this simple case.

\begin{theorem}
\label{Thm:BoundOf3}
If $n \geq 4$, then $\gon(G_n) \geq 3$.
\end{theorem}

\begin{proof}
We show that for $n \geq 4$, there does not exist an element $D \in \cA(G_n)$ such that $D-x \in \cA(G_n)$ for all $x \in \cA(G_n)$.  It will then follow from Corollary~\ref{Cor:Gonality} that $\gon(G_n) \geq 3$ for $n \geq 4$.  We break this into cases.
 \begin{enumerate}
        \item If $D = F_{2n-1}$, then $D-F_{2n-2} = F_{2n-3}$. Since $F_{2n-3} \notin \{0, F_{2n-1}\}$, and the leading term of its Zeckendorf form is not $F_{2n-2}$, we have $D-F_{2n-2} = F_{2n-3} \notin \cA(G_n)$ by Corollary~\ref{Cor:A}.
        
        \item If $D = 0$, then $D - 2F_{2n-3} = F_{2n-2} + F_{2n-4}$.  By Corollary~\ref{Cor:A}, the second largest even-index Fibonacci number of an element of $\cA(G_n)$ is $F_{2n-6}$.  As such, $D - 2F_{2n-3} \notin \cA(G_n)$.
        
        \item Otherwise, by Corollary~\ref{Cor:A}, $D$ has leading term $F_{2n-2}$ in its Zeckendorf form.  Thus, by Corollary~\ref{Cor:A} again, we have $D - F_{2n-2} \in \cA(G_n)$ if and only if either $D - F_{2n-2} = 0$ or $D -F_{2n-2} = F_{2n-1}$.  In the first case, we see that $D$ is equal to $F_{2n-2}$, and the second, we see that $D$ is equal to 0.  We have already considered the case where $D=0$.  If $D = F_{2n-2}$, then $D - F_{2n-1} = F_{2n-1} + F_{2n-4}$, which is not in $\cA(G_n)$ by Corollary~\ref{Cor:A}.
    \end{enumerate}
\end{proof}

In Theorem~\ref{Thm:BoundOf4} below, we will prove that $\gon(G_n) \geq 4$ when $n \geq 6$.  Because $\mathrm{sn}(G_n) = 3$, the scramble number is insufficient to compute this bound.  Instead, we will argue in a similar way to the proof of Theorem~\ref{Thm:BoundOf3}.

\subsection{The Set $2\cA(G_n)$}
\label{Sec:2A}
In this section, we describe the Zeckendorf form of all elements of $2\cA(G_n)$.  Describing this set will require us to write down the Zeckendorf form of the sum of two numbers.  This is a component of \emph{Zeckendorf arithmetic}, as described in \cite{Freitag1998, Fenwick03}.  A key idea, used implicitly in all the proofs of this section, is that if $x$ and $y$ are both smaller than $F_k$, then $x+y$ is smaller than $F_{k+2}$.  Thus, if we know only the first few terms of the Zeckendorf forms of $x$ and $y$, we can compute the leading terms of the Zeckendorf form of $x+y$.

By Corollary~\ref{Cor:A}, most elements of $2\cA(G_n)$ are of the following form.

\begin{lemma}
\label{Lem:MostOf2A}
Let $2 \leq a \leq b \leq n$.  Then the leading term of the Zeckendorf form of $F_{a}F_{2n-a} + F_{b}F_{2n-b}$ is $F_{2n-1}$, followed by either $F_{2n-3}$ or $F_{2n-4}$.  Moreover:
\begin{enumerate}
\item  if the leading terms are $F_{2n-1} + F_{2n-3} + F_{2n-5}$, then $a = b = 3$,
\item  if the leading terms are $F_{2n-1} + F_{2n-4}$, then either $a=b=2$ or the next term is $F_{2n-6}$; if there is another term, it is at most $F_{2n-9}$, and
\item if the leading terms are $F_{2n-1}+F_{2n-3}+F_{2n-6}$, then $a=3, b>3$ and next term is at most $F_{2n-9}$.
\end{enumerate}
\end{lemma}

\begin{proof}
By Corollary~\ref{Cor:A} we have
\[
F_{a}F_{2n-a} + F_{b}F_{2n-b} = \sum_{r=1}^{\lfloor \frac{a}{2}\rfloor} F_{2n-4r+2}+ \sum_{r=1}^{\lfloor \frac{b}{2}\rfloor}F_{2n-4r+2}+\left(\varepsilon_a F_{2n-2a+1}+\varepsilon_b F_{2n-2b+1}\right),
\]
where
\[
\varepsilon_a = \begin{cases}
        0 & \text{ if $a$ is even } \\
        1& \text{ if $a$ is odd.}
    \end{cases}
\]
Combining the like terms, the above is equal to
\begin{align*}
&= \sum_{r=1}^{\lfloor \frac{a}{2}\rfloor} 2F_{2n-4r+2} + \sum_{r=\lfloor\frac{a}{2}\rfloor + 1 }^{\lfloor \frac{b}{2}\rfloor} F_{2n-4r+2} +\left(\varepsilon_a F_{2n-2a+1} + \varepsilon_b F_{2n-2b+1} \right)\\
&= \sum_{r=1}^{\lfloor \frac{a}{2}\rfloor} (F_{2n-4r+3} + F_{2n-4r}) + \sum_{r=\lfloor\frac{a}{2}\rfloor + 1}^{\lfloor \frac{b}{2}\rfloor} F_{2n-4r+2} +\left(\varepsilon_a F_{2n-2a+1} + \varepsilon_b F_{2n-2b+1}\right).
\end{align*}
If $a=2$ and $b=3$, then the expression above is $F_{2n-1} + F_{2n-4} + F_{2n-5} = F_{2n-1} + F_{2n-3}$.  If $a=b=3$, then it is equal to $F_{2n-1} + F_{2n-4} + 2F_{2n-5} = F_{2n-1} + F_{2n-3} + F_{2n-5}$.  If $a=2$ and $b\neq 3$, the leading terms of the above expression are $F_{2n-1} + F_{2n-4}$, and if $b>3$ the next term is $F_{2n-6}$. Moreover, if the next possible largest term is $F_{2n-9}$. Similarly, if $a=3$ and $b>3$, then the leading terms of the above expression are $F_{2n-1} + F_{2n-4} + F_{2n-5} +F_{2n-6} = F_{2n-1} + F_{2n-3} + F_{2n-6}$. The next possible largest term is $F_{2n-9}$.

Otherwise, we have that the left hand sum has leading terms $F_{2n-1} + F_{2n-4} + F_{2n-5} + F_{2n-8}$, which becomes $F_{2n-1} + F_{2n-3} + F_{2n-8}$ in Zeckendorf form.  Because the righthand sum has leading term at most $F_{2n-9}$, it does not affect the first two leading terms, so in this case, we have leading terms $F_{2n-1} + F_{2n-3}$.
\end{proof}

We now turn to the elements of $2\cA(G_n)$ that are not described by Lemma~\ref{Lem:MostOf2A}.

\begin{lemma}
\label{Lem:AddF2n-3}
If $a\geq 4$, the leading term of $F_{2n-1} + F_{a}F_{2n-a}$ in Zeckendorf form is $F_{2n-6}$, with the next possible leading term either $F_{2n-9}$ or $F_{2n-10}$.
\end{lemma}

\begin{proof}
Let $a\geq 4$.  We note
\begin{align*}
    F_{2n-1} + F_{a}F_{2n-a} &= F_{2n-1} + \left(\sum_{r=1}^{\lfloor \frac{a}{2}\rfloor} F_{2n-4r+2} \right) +\varepsilon_a F_{2n-2a+1}\\
    &= \left(\sum_{r=2}^{\lfloor \frac{a}{2}\rfloor} F_{2n-4r+2} \right) + \varepsilon_a F_{2n-2a+1} \pmod{F_{2n}}.
    \end{align*}
If $a=4$, then $\varepsilon_a = 0$, and if $a \geq 5$, then $F_{2n-2a+1}$ is at most $F_{2n-9}$.  As such, the leading term is unaffected by this term, so the leading term is $F_{2n-6}$. Moreover, if $a=5$, the second leading term is $F_{2n-9}$. Otherwise, we have that the second leading term is $F_{2n-10}$, since we obtain the term with $r = 3$.
\end{proof}

Below we list the remaining elements of $2\cA(G_n)$ not covered by one of the previous cases.
\begin{lemma}
\label{Lem:Remaining2A}
We have:
\begin{itemize}
    \item $2F_{2n-1} = F_{2n-3} \pmod{F_{2n}}$
    \item $F_{2n-1} + F_{2n-2} = 0 \pmod{F_{2n}}$
    \item $F_{2n-1} + 2F_{2n-3} = F_{2n-5} \pmod{F_{2n}}.$
\end{itemize}
\end{lemma}

We summarize the results of this subsection in the following corollary.

\begin{corollary}
\label{Cor:2A}
Let $D \in 2\cA(G_n)$.  Then either:
\begin{enumerate}
\item  $D \in \cA(G_n)$,
\item  $D = F_{2n-3}$,
\item  $D = F_{2n-5}$,
\item  $D$ has leading term $F_{2n-6}$, followed by either $F_{2n-9}$ or $F_{2n-10}$.
\item  $D$ has leading term $F_{2n-1}$, followed by either $F_{2n-3}$ or $F_{2n-4}$.  Moreover:
\begin{enumerate}
\item  if the leading terms are $F_{2n-1} + F_{2n-3} + F_{2n-5}$, then $D = F_{2n-1} + F_{2n-3} + F_{2n-5}$, and
\item  if the leading terms are $F_{2n-1} + F_{2n-4}$, then either $D = F_{2n-1} + F_{2n-4}$ or the next term is $F_{2n-6}$; if there is another term, it is at most $F_{2n-9}.$
\item if the leading terms are $F_{2n-1}+F_{2n-3}+F_{2n-6}$, then the next term is at most $F_{2n-9}$.
\end{enumerate}
\end{enumerate}
\end{corollary}

The following lemma will not be used in this section, but we will need it in Section~\ref{Sec:MainThm}.

\begin{lemma}
\label{Lem:Sp case 1}
    If $D\in 2 \cA(G_n)$ has leading terms $F_{2n-1} + F_{2n-3} + F_{2n-8}$, then the next possible leading term is $F_{2n-10}$.
\end{lemma}

\begin{proof}
From the proof of Lemma~\ref{Lem:MostOf2A}, we have that elements in $D\in 2 \cA(G_n)$ have possible leading terms $F_{2n-1} + F_{2n-3} + F_{2n-8}$ if they are the sum of two elements $D_1+D_2$ with leading terms $F_{2n-2} + F_{2n-6}$. If either has next term $F_{2n-9}$, then we have that the leading terms of $D$ would be $F_{2n-1} + F_{2n-3} + F_{2n-7}$ or $F_{2n-1} + F_{2n-3} + F_{2n-6}$.  In a similar manner, if both have next term $F_{2n-10}$, then $D$ would have leading terms $F_{2n-1} + F_{2n-3} + F_{2n-7}$.  As such, the only remaining case is $D_1 = F_{2n-2} + F_{2n-6}$ and, $D_2 = D_1$ or $D_2$ has leading terms $F_{2n-2} + F_{2n-6} + F_{2n-10}$.  Thus, if $D$ has leading terms $F_{2n-1} + F_{2n-3} + F_{2n-8}$, then the next possible leading term is $F_{2n-10}$.
\end{proof}

%\begin{lemma}
%\label{Lem:Sp case 2}
%If $D \in 2\cA(G_n)$ has leading terms $F_{2n-1} + F_{2n-3} + F_{2n-7} + F_{2n-9}$, then $D$ is in fact equal to $F_{2n-1} + F_{2n-3} + F_{2n-7} + F_{2n-9}$.
%\end{lemma}

%\begin{proof}
%From the proof of Lemma~\ref{Lem:MostOf2A}, we have that elements in $D\in 2 \cA(G_n)$ have possible leading terms $F_{2n-1} + F_{2n-3} + F_{2n-7}$ if they are the sum of two elements $D_1+D_2$ with leading terms $F_{2n-2} + F_{2n-6}$. There are two possible cases:
    
%    \begin{enumerate}
    
%    \item If both $D_1, D_2$ have third term $F_{2n-10}$, then $D_1 + D_2$ has leading terms $F_{2n-1} + F_{2n-3} + F_{2n-7} + F_{2n-12}$ or $F_{2n-1} + F_{2n-3} + F_{2n-7} + F_{2n-11}$.
    
%    \item Suppose that at least one of $D_1, D_2$ has third term $F_{2n-9}$.  In this case, if both have third term $F_{2n-9}$, then $D = F_{2n-1} + F_{2n-3} + F_{2n-7} + F_{2n-9}$. If only one of them has third term $F_{2n-9}$, we have either $D = F_{2n-1} + F_{2n-3} + F_{2n-7}$ or $D$ has leading terms $F_{2n-1} + F_{2n-3} + F_{2n-7} + F_{2n-10}$.
%    \end{enumerate}
    
%    Therefore, if $D \in 2\cA(G_n)$ has leading terms $F_{2n-1} + F_{2n-3} + F_{2n-7} + F_{2n-9}$, then there can be no other leading term.
%\end{proof}

\subsection{A Stronger Lower Bound on the Gonality of Strip Graphs}
\label{Sec:StripStrongerBound}

To compute the gonality of $G_n$, it will be helpful to use the Zeckendorf representation of some elements of $-\cA(G_n) \pmod{F_{2n}}$. Below we leave a table:

\begin{center}
\begin{tabular}{|c|c|}
\hline 
   $D$  & $-D$ \\
   \hline
   $F_{2n-1}$  & $F_{2n-2}$ \\
   \hline
   $F_{2n-2}$ & $F_{2n-1}$\\
   \hline
   $2F_{2n-3}$ & $F_{2n-2} + F_{2n-4}$\\
   \hline
   $3F_{2n-4}$ & $F_{2n-2} + F_{2n-4} + F_{2n-7}$\\
   \hline 
   $5F_{2n-5}$ & $F_{2n-2} + F_{2n-4} + F_{2n-8}$\\
   \hline
   $8F_{2n-6}$ & $F_{2n-2} + F_{2n-4} + F_{2n-8} + F_{2n-11}$\\
   \hline 
   $13F_{2n-7}$ & $F_{2n-2} + F_{2n-4} + F_{2n-8} + F_{2n-12}$\\
   \hline
\end{tabular}
\end{center}

We now show that $\gon(G_n) \geq 4$ when $n \geq 6$.  Our approach will be similar to that of Theorem~\ref{Thm:BoundOf4}, using our description of the set $2\cA(G_n)$.

\begin{theorem}
\label{Thm:BoundOf4}
If $n \geq 6$, then $\gon(G_n) \geq 4$.
\end{theorem}

\begin{proof}
We show that, for $n \geq 6$, there does not exist an element $D \in 2\cA(G_n)$ such that $D-x \in 2\cA(G_n)$ for all $x \in \cA(G_n)$.  It will then follow from Corollary~\ref{Cor:Gonality} that $\gon(G_n) \geq 4$ for $n \geq 6$.  We again break this into cases.  The first 5 cases cover elements of $\cA(G_n)$, the next 2 cover elements of $F_{2n-1} + \cA(G_n)$, and the last case covers all the remaining elements of $2\cA(G_n)$.

\begin{enumerate}

    \item If $D=0$, then $D - 2F_{2n-3} = F_{2n-2} + F_{2n-4}$.  By Corollary~\ref{Cor:2A}, if $D \in 2\cA(G_n)$ has leading term $F_{2n-2}$, then $D \in \cA(G_n)$.  By Corollary~\ref{Cor:A}, if an element has leading term $F_{2n-2}$, then the next largest even-index term is $F_{2n-6}$.  As such, no element in $2\cA(G_n)$ has leading terms $F_{2n-2} + F_{2n-4}$.  Hence $D - 2F_{2n-3} \notin 2\cA(G_n).$

    \item If $D = F_{2n-1}$, then $D - 2F_{2n-3} = D + F_{2n-2} + F_{2n-4} = F_{2n-4}$.  By Corollary~\ref{Cor:2A}, no element in $2\cA(G_n)$ has leading term $F_{2n-4}$, so $D - 2F_{2n-3} \notin 2\cA(G_n).$

    \item  If $D = F_{2n-2}$, then $D - 5F_{2n-5}$ has leading terms $F_{2n-1} + F_{2n-3} + F_{2n-6} + F_{2n-8}$. By Corollary~\ref{Cor:2A}, if an element has these first 3 leading terms in $2\cA(G_n)$, then the next term is at most $F_{2n-9}$. Since $F_{2n-8}>F_{2n-9}$, we see that $D-5F_{2n-5}$ is not in $2\cA(G_n)$.

    \item  If $D = 2F_{2n-3}$, then $D - 3F_{2n-4} = F_{2n-7}$, which is not in $2\cA(G_n)$ by Corollary~\ref{Cor:2A}.

    \item If $D \in \cA(G_n) \smallsetminus \{0, F_{2n-1}, F_{2n-2}, 2F_{2n-3} \}$, then $D = F_{k}F_{2n-k}$ with $k \geq 4$.  By Corollary~\ref{Cor:A}, its leading terms are $F_{2n-2} + F_{2n-6}$. As such, $D - 2F_{2n-3} = D + F_{2n-2} + F_{2n-4}$ has leading terms $F_{2n-1} + F_{2n-3} + F_{2n-5}$. By Corollary~\ref{Cor:2A}, the only element with these leading terms in $2\cA(G_n)$ is $4F_{2n-3}$.  If this is the case, then $D = 6F_{2n-3}$.  However, $6F_{2n-3} = F_{2n-2} + F_{2n-7}$, which is not in $2\cA(G_n)$ by Corollary~\ref{Cor:2A}.
    
    \item If $D = 2F_{2n-1} = F_{2n-3}$, then $D - 3F_{2n-4} = D + F_{2n-2} + F_{2n-4} + F_{2n-7} = F_{2n-1} + F_{2n-4} + F_{2n-7}$.  By Corollary~\ref{Cor:2A}, if an element has leading terms $F_{2n-1} + F_{2n-4}$, then the next term is $F_{2n-6}$. Thus, $D - 3F_{2n-4} \notin 2\cA(G_n)$.

    \item If $D = F_{2n-1} + F_{k}F_{2n-k}$ with $k \geq 3$, then by Lemma~\ref{Lem:AddF2n-3}, $D - F_{2n-2}$ has leading terms $F_{2n-1} + F_{2n-5}$ (if $k=3$) or $F_{2n-1} + F_{2n-6}$ (if $k \geq 4)$.  By Corollary~\ref{Cor:2A},  there is no element in $2\cA(G_n)$ with these leading terms.

    \item  Finally, assume that $D$ is a sum of two elements of $\cA(G_n)$, neither of which is $0$ or $F_{2n-1}$.  By Lemma~\ref{Lem:MostOf2A}, $D$ has leading terms $F_{2n-1} + F_{2n-3}$ or $F_{2n-1} + F_{2n-4}$. Then $D - F_{2n-1}$ has leading term $F_{2n-3}$ or $F_{2n-4}$.  By Corollary~\ref{Cor:2A}, there is no element of $2\cA(G_n)$ with leading term $F_{2n-4}$, and the only element of $2\cA(G_n)$ with leading term $F_{2n-3}$ is $F_{2n-3}$ itself.  It follows that $D = F_{2n-1} + F_{2n-3}$. However, $D-3F_{2n-4} = F_{2n-2} + F_{2n-7}$. Hence $D-3F_{2n-4} \notin 2\cA(G_n)$ by Corollary~\ref{Cor:2A}.
    \end{enumerate}
\end{proof}

\section{The Gonality of Strip Graphs}
\label{Sec:MainThm}

In this section, we prove Theorem~\ref{Thm:MainThm}.  Our proof follows the same strategy as that of Theorems~\ref{Thm:BoundOf3} and~\ref{Thm:BoundOf4}, though there are many more cases.  We start by showing that $\gon (G_n) \leq 5$ for all $n$.

\subsection{A Divisor of Rank 1}
\label{Sec:UpperBound}
We now find a divisor of degree 5 and rank at least 1 on the strip graph $G_n$.  This shows that the gonality of $G_n$ is at most 5.

\begin{lemma}
\label{Lem:UpperBound}
For $3 \leq k \leq n$, we have
\[
F_{2n} + 2F_{2n-1} - F_{k}F_{2n-k} = F_{k-2}F_{2n-k+2} + 3F_{k-1}F_{2n-k+1}.
\]
\end{lemma}

\begin{proof}
We fix $n$ and prove this by induction on $k$.  For the base case $k=3$, we have
\[
F_{2n} + 2F_{2n-1} - 2F_{2n-3} = 2F_{2n-1} + 2F_{2n-2} - F_{2n-3} = F_{2n-1} + 3F_{2n-2}.
\]
Now, assume the equation holds for $k$.  We will prove it for $k+1$ by the following tedious calculation.
\begin{align*}
&\mathcolor{white}{=} F_{2n} + 2F_{2n-1} - F_{k+1}F_{2n-k-1} \\
&= F_{2n} + 2F_{2n-1} - F_{k}F_{2n-k-1} - F_{k-1}F_{2n-k-1} \\
	&= F_{2n} + 2F_{2n-1} + F_{k}F_{2n-k-2} - F_{k}F_{2n-k} - F_{k-1}F_{2n-k-1}\\
	&= F_{k-2}F_{2n-k+2} + 3F_{k-1}F_{2n-k+1} + F_{k}F_{2n-k-2} - F_{k-1}F_{2n-k-1}
\end{align*}
where the last equality holds by inductive hypothesis.  Now, the above is equal to
\begin{align*}
	&\mathcolor{white}{=} 3F_{k-1}F_{2n-k+1} + F_{k-2}F_{2n-k+1} + F_{k-2}F_{2n-k} + F_{k}F_{2n-k-2} - F_{k-1}F_{2n-k-1}\\
	&= 2F_{k-1}F_{2n-k+1} + F_{k}F_{2n-k+1} + F_{k-2}F_{2n-k} + F_{k}F_{2n-k-2} - F_{k-1}F_{2n-k-1}\\
	&= 2F_{k-1}F_{2n-k+1} + F_{k}F_{2n-k} + F_{k}F_{2n-k-1} + F_{k-2}F_{2n-k} + F_{k}F_{2n-k-2} - F_{k-1}F_{2n-k-1}\\
	&= 2F_{k-1}F_{2n-k+1} + 2F_{k}F_{2n-k} + F_{k-2}F_{2n-k} - F_{k-1}F_{2n-k-1}\\
	&= F_{k-1}F_{2n-k+1} + 2F_{k}F_{2n-k} + F_{k-1}F_{2n-k} + F_{k-1}F_{2n-k-1} + F_{k-2}F_{2n-k} - F_{k-1}F_{2n-k-1}\\
	&= F_{k-1}F_{2n-k+1} + 2F_{k}F_{2n-k} + F_{k-1}F_{2n-k} + F_{k-2}F_{2n-k}\\
	&= F_{k-1}F_{2n-k+1} + 3F_{k}F_{2n-k}\\
\end{align*}
\end{proof}

This yields the following result.

\begin{lemma}
\label{Lem:DivisorDegree5}
If $D = 3v_0+2v_1$, then $D$ has rank at least 1.
\end{lemma}

\begin{proof}
By Proposition~\ref{Prop:Rank}, it suffices to show that $D-5v_0 - \cA(G_n) \subseteq 4\cA(G_n)$.  Note that $D-5v_0 = 2v_1 - 2v_0$, so by Lemma~\ref{Lem:FibProduct} we have $\varphi (D-5v_0) = 2F_{2n-1}$.  By Lemma~\ref{Lem:UpperBound}, we see that $2F_{2n-1} - F_{k}F_{2n-k} \in 4\cA(G_n)$ for all $k \geq 3$.  It therefore suffices to check the cases where $k \leq 2$.  For these cases, we have
\begin{align*}
2F_{2n-1}-0 = 2F_{2n-1} &\in 2\cA(G_n) \subseteq 4\cA(G_n) \\ 
2F_{2n-1}-F_{2n-1} = F_{2n-1} &\in \cA(G_n) \subseteq 4\cA(G_n) \mbox{ and } \\
2F_{2n-1}-F_{2n-2} = F_{2n-1}-F_{2n-3} = F_{2n-2} + 2F_{2n-3} &\in 3\cA(G_n) \subseteq 4\cA(G_n).
\end{align*}
\end{proof}

Lemma~\ref{Lem:DivisorDegree5} demonstrates that the gonality of $G_n$ is at most 5.  One can also check that the divisor $D = 3v_0 + 2v_1$ has positive rank ``by hand'', using Dhar's burning algorithm to compute the $v_i$-reduced divisor equivalent to $D$ for all $i$.  We have chosen to avoid this in order to emphasize our approach.

\subsection{The Set $3A$}

In this section, we describe the Zeckendorf form of all elements of $3\cA(G_n)$.  The approach here is similar to that of Section~\ref{Sec:2A}.  In particular, we use Corollary~\ref{Cor:A} to describe the Zeckendorf form of every element of $\cA(G_n)$, and Corollary~\ref{Cor:2A} to describe that of every element of $2\cA(G_n)$, and then use Zeckendorf arithmetic to describe their sum.

\begin{lemma}
    If $D\in 2\cA(G_n)$ has leading terms $F_{2n-1} + F_{2n-4}$, then $D + F_{2n-1} \in \cA(G_n)$.
\end{lemma}

\begin{proof}
    By the proof of Lemma~\ref{Lem:MostOf2A},  $D \in F_{2n-2} + \cA(G_n)$, hence $D + F_{2n-1} \in \cA(G_n)$.
\end{proof}

\begin{lemma}
\label{Lem:FirstCase}
    If $D \in 2\cA(G_n)$ has leading terms $F_{2n-1} + F_{2n-4}$ and $D' \in \cA(G_n)$ has leading term $F_{2n-2}$, then either:
    \begin{enumerate}
        \item $D+D' = F_{2n-4}$,
        \item $D+D'$ has leading terms $F_{2n-4} + F_{2n-6}$, with next possible term $F_{2n-9}$ or $F_{2n-10}$,
        \item $D+D'$ has leading terms $F_{2n-3} + F_{2n-6}$, with next possible term $F_{2n-9}$ or $F_{2n-10}$,
        \item $D+D'$ has leading terms $F_{2n-3} + F_{2n-7}$, or
        \item $D+D'$ has leading terms $F_{2n-3} + F_{2n-8}$.
    \end{enumerate}
\end{lemma}

\begin{proof}
By Lemma~\ref{Lem:MostOf2A}, either $D = F_{2n-1} + F_{2n-4}$, or the next leading term is $F_{2n-6}$.  By Corollary~\ref{Cor:A}, either $D' = F_{2n-2}$, or the next leading term is $F_{2n-5}$, or $F_{2n-6}$.  The proof then follows by case analysis, for each possible combination of leading terms of $D$ and $D'$.
\end{proof}

The remaining lemmas in this section follow from a straightforward case analysis, similar to that of Lemma~\ref{Lem:FirstCase}.  We omit the details.

\begin{lemma}
    If $D \in 2\cA(G_n)$ has leading terms $F_{2n-1} + F_{2n-3}$, then either
    \begin{enumerate}
        \item $D + F_{2n-1} = F_{2n-3} + F_{2n-4} + F_{2n-7}$,
        \item $D + F_{2n-1}$ has leading terms $F_{2n-2} + F_{2n-4}$, with next possible term smaller than $F_{2n-9}$,
        \item $D + F_{2n-1}$ has leading terms $F_{2n-2} + F_{2n-5} + F_{2n-7}$, or
        \item $D + F_{2n-1}$ has leading terms $F_{2n-2} + F_{2n-5} + F_{2n-8}$, with next possible term $F_{2n-10}$.
    \end{enumerate}
\end{lemma}

\begin{lemma}
    If $D \in 2\cA(G_n)$ has leading terms $F_{2n-1} + F_{2n-3}$, then either:
    \begin{enumerate}
        \item $D + F_{2n-2} = F_{2n-3} + F_{2n-5}$,
        \item $D + F_{2n-2}$ has leading terms $F_{2n-3} + F_{2n-6}$, with next possible term $F_{2n-9}$ or $F_{2n-10}$,
        \item $D + F_{2n-2}$ has leading terms $F_{2n-3} + F_{2n-7}$, or
        \item $D + F_{2n-2}$ has leading terms $F_{2n-3} + F_{2n-8}$.
    \end{enumerate}
\end{lemma}

\begin{lemma}
    If $D \in 2\cA(G_n)$ has leading terms $F_{2n-1} + F_{2n-3}$, then either:
    \begin{enumerate}
        \item $D + F_{2n-2} + F_{2n-5} = F_{2n-2} + F_{2n-7}$,
        \item $D + F_{2n-2} + F_{2n-5}$ has leading terms $F_{2n-2} + F_{2n-10}$,
        \item $D + F_{2n-2} + F_{2n-5}$ has leading terms $F_{2n-3} + F_{2n-5} + F_{2n-7}$, or
        \item $D + F_{2n-2} + F_{2n-5}$ has leading terms $F_{2n-3} + F_{2n-5} + F_{2n-8}$.
    \end{enumerate}
\end{lemma}

\begin{lemma}
    If $D \in 2\cA(G_n)$ has leading terms $F_{2n-1} + F_{2n-3}$ and $D' \in \cA(G_n)$ has leading terms $F_{2n-2} + F_{2n-6}$, then either:
    \begin{enumerate}
        \item $D+D' = F_{2n-2} + F_{2n-9}$,
        \item $D+D'$ has leading terms $F_{2n-2} + F_{2n-10}$,
        \item $D+D'$ has leading terms $F_{2n-3} + F_{2n-5}$,
        \item $D+D'$ has leading terms $F_{2n-3} + F_{2n-5} + F_{2n-8}$, or
        \item $D+D'$ has leading terms $F_{2n-3} + F_{2n-6} + F_{2n-8}$.
    \end{enumerate}
\end{lemma}

\begin{lemma}
    If $D' \in \cA(G_n)$, then either:
    \begin{enumerate}
        \item $F_{2n-3} + D' = F_{2n-1} + F_{2n-5}$, or
        \item $F_{2n-3} + D'$ has leading terms $F_{2n-1} + F_{2n-6}$, with next possible term $F_{2n-9}$ or $F_{2n-10}$.
    \end{enumerate}
\end{lemma}

\begin{lemma}
    If $D' \in \cA(G_n)$, then either:
    \begin{enumerate}
        \item $F_{2n-5} + D' = F_{2n-2} + F_{2n-5}$, or 
        \item $F_{2n-5} + D'$ has leading terms $F_{2n-2} + F_{2n-4}$, with next possible term $F_{2n-7}$, $F_{2n-9}$ or $F_{2n-10}$.
    \end{enumerate}
\end{lemma}

\begin{lemma}
\label{Lem:LastCase}
    If $D \in 2\cA(G_n)$ has leading term $F_{2n-6}$ and $D' \in \cA(G_n)$, then either:
    \begin{enumerate}
        \item $D+D' = F_{2n-1} + F_{2n-6}$,
        \item $D+D' = F_{2n-2} + F_{2n-4}$, or
        \item $D+D'$ has leading terms $F_{2n-2} + F_{2n-5}$, with next possible term $F_{2n-7}$ or $(F_{2n-8}+F_{2n-10})$.
    \end{enumerate}
\end{lemma}

We summarize the results of Lemmas~\ref{Lem:FirstCase}-\ref{Lem:LastCase} as follows:

\begin{corollary}
\label{Cor:3A}
Let $D \in 3\cA(G_n) \smallsetminus 2\cA(G_n)$.
\begin{enumerate}
\item  If $D$ has leading term $F_{2n-4}$, then either:
    \begin{enumerate}
    \item  $D = F_{2n-4}$, or
    \item  $D$ has leading terms $F_{2n-4} + F_{2n-6}$, followed by either $F_{2n-9}$ or $F_{2n-10}$.
    \end{enumerate}
\item  If $D$ has leading term $F_{2n-3}$, then either:
    \begin{enumerate}
    \item  $D$ has leading terms $F_{2n-3} + F_{2n-8}$,
    \item  $D$ has leading terms $F_{2n-3} + F_{2n-7}$,
    \item  $D$ has leading terms $F_{2n-3} + F_{2n-6}$,
    \item  $D = F_{2n-3} + F_{2n-5}$,
    \item  $D$ has leading terms $F_{2n-3} + F_{2n-5} + F_{2n-8}$, or
    \item  $D$ has leading terms $F_{2n-3} + F_{2n-5} + F_{2n-7}$.
    \end{enumerate}
\item  If $D$ has leading term $F_{2n-2}$, then either:
    \begin{enumerate}
    \item  $D$ has leading terms $F_{2n-2} + F_{2n-10}$,
    \item  $D = F_{2n-2} + F_{2n-9}$,
    \item  $D = F_{2n-2} + F_{2n-7}$,
    \item  $D = F_{2n-2} + F_{2n-5}$,
    \item  $D$ has leading terms $F_{2n-2} + F_{2n-5} + F_{2n-8}$ with next possible term $F_{2n-10}$,
    \item  $D$ has leading terms $F_{2n-2} + F_{2n-5} + F_{2n-7}$,
    \item  $D = F_{2n-2} + F_{2n-4}$, or
    \item  $D$ has leading terms $F_{2n-2} + F_{2n-4}$, followed by either $F_{2n-7}$ or a term smaller than $F_{2n-8}$.
    \end{enumerate}
\item  If $D$ has leading term $F_{2n-1}$, then either:
    \begin{enumerate}
    \item  $D$ has leading terms $F_{2n-1} + F_{2n-6}$, followed by either $F_{2n-9}$ or $F_{2n-10}$,
    \item  $D = F_{2n-1} + F_{2n-6}$, or
    \item  $D = F_{2n-1} + F_{2n-5}$.
    \end{enumerate}
\end{enumerate}
\end{corollary}

\subsection{Proof of Theorem~\ref{Thm:MainThm}}

We now complete our proof of the gonality of $G_n$.

\begin{proof}[Proof of Theorem~\ref{Thm:MainThm}]
By Lemma~\ref{Lem:DivisorDegree5}, we have $\gon (G_n) \leq 5$.  By \cite{AtanasovRanganathan}, the Brill-Noether existence conjecture holds for all graphs of genus at most 5.  As a consequence, $\gon(G_n) \leq \lceil \frac{n}{2} \rceil$ for $n \leq 6$.  When $n=7$, we note that the divisor $2v_4 + 2v_5$ has positive rank, hence $\gon(G_7) \leq 4$.

For $n \leq 7$, it therefore suffices to show that $\gon (G_n) \geq \lceil \frac{n+1}{2} \rceil$.  If $n=0$ or $1$, then $G_n$ is a tree, and hence has gonality 1.  If $n=2$ or $3$, then then $G_n$ is not a tree, hence it has gonality at least 2.  If $n=4$ or $5$, then $\gon (G_n) \geq 3$ by Theorem~\ref{Thm:BoundOf3}, and if $n=6$ or $7$, then $\gon (G_n) \geq 4$ by Theorem~\ref{Thm:BoundOf4}.

For the rest of the proof, assume that $n \geq 8$.  We will show that there does not exist a $D \in 3\cA(G_n)$ such that $D-x \in 3\cA(G_n)$ for all $x \in A$.  It will then follow from Corollary~\ref{Cor:Gonality} that $\gon(G_n) \geq 5$ for $n \geq 8$.  As in the proofs of Theorem~\ref{Thm:BoundOf3} and Theorem~\ref{Thm:BoundOf4}, we shall proceed by examining the different possible leading terms of $D$.  Although there are many cases, nearly all are proven by straightforward calculations.

\begin{enumerate}

    \item Let $D$ have leading term $F_{2n-6}$. Then $D - 2F_{2n-3} = D + F_{2n-2} + F_{2n-4}$ has leading terms $F_{2n-2} + F_{2n-4} + F_{2n-6}$. By Corollary~\ref{Cor:3A} there is no such case, so $D - 2F_{2n-3} \notin 3\cA(G_n)$.

    \item Let $D$ have leading term $F_{2n-5}$.  By Corollary~\ref{Cor:3A}, $D$ must be in $2\cA(G_n)$, and by Corollary~\ref{Cor:2A}, we see that $D = F_{2n-5}$.  Note that $D - 3F_{2n-4} = F_{2n-1} + F_{2n-7}$.  By Corollary~\ref{Cor:3A}, no element in $3\cA(G_n)$ has such leading terms, so $D - 3F_{2n-4} \notin 3\cA(G_n)$.

    \item  Let $D$ have leading term $F_{2n-4}$.  We break this into subcases:

    \begin{enumerate}

        \item Let $D$ have leading terms $F_{2n-4} + F_{2n-6}$. Then $D - 2F_{2n-3}$ has leading terms $F_{2n-1} + F_{2n-5} + F_{2n-8}$. By Corollary~\ref{Cor:3A}, the only element in $3\cA(G_n)$ with leading terms $F_{2n-1} + F_{2n-5}$ is $F_{2n-1} + F_{2n-5}$ itself. As such, $D - 2F_{2n-3} \notin 3\cA(G_n)$.

        \item Let $D = F_{2n-4}$.  Then $D-5F_{2n-5}$ 
        has leading terms $F_{2n-1}+F_{2n-6}+F_{2n-8}$. By Corollary~\ref{Cor:3A}, we have that $D - 5F_{2n-5} \notin 3\cA(G_n)$.
        \end{enumerate}

    \item  Let $D$ has leading term $F_{2n-3}$.  We again break this into subcases.

    \begin{enumerate}

        \item Let $D$ have leading terms $F_{2n-3} + F_{2n-8}$.  Then $D - F_{2n-1}$ has leading terms $F_{2n-1} + F_{2n-8}$, which is not in $3\cA(G_n)$ by Corollary~\ref{Cor:3A}.

        \item Let $D$ have leading terms $F_{2n-3} + F_{2n-7}$.  Then $D - F_{2n-1}$ has leading terms $F_{2n-1} + F_{2n-7}$, which is not in $3\cA(G_n)$ by Corollary~\ref{Cor:3A}.

        \item Let $D$ have leading terms $F_{2n-3} + F_{2n-5}$.  Then $D-F_{2n-2}$ has leading terms $F_{2n-1} + F_{2n-3} + F_{2n-5}$.  By Corollary~\ref{Cor:3A}, if $D - F_{2n-2} \in 3\cA(G_n)$, we see that it must in fact be in $2\cA(G_n)$.  By Lemma~\ref{Lem:MostOf2A}, it follows that $D - F_{2n-2} = F_{2n-1} + F_{2n-3} + F_{2n-5}$.  In particular, $D = F_{2n-3} + F_{2n-5}$.  Then $D - 8F_{2n-6} = F_{2n-1} + F_{2n-3} + F_{2n-8} + F_{2n-11}$.  If this element is in $3\cA(G_n)$, then it is in $2\cA(G_n)$ by Corollary~\ref{Cor:3A}, but this is not the case by Lemma~\ref{Lem:Sp case 1}.

        \item Let $D$ have leading terms $F_{2n-3} + F_{2n-6}$.  We must consider further subcases.
        
            \begin{enumerate}
            \item  If $D = F_{2n-3} + F_{2n-6}$, then $D - 5F_{2n-5} = F_{2n-1} + F_{2n-4} + F_{2n-6} + F_{2n-8}$, which is not in $3\cA(G_n) \setminus 2\cA(G_n)$ by Corollary~\ref{Cor:3A}. Moreover, by Corollary~\ref{Cor:2A} since $F_{2n-8}>F_{2n-9}$, $D-5F_{2n-5} \notin 2\cA(G_n)$.
        
            \item  If $D$ has leading terms $F_{2n-3} + F_{2n-6} + F_{2n-8}$, then $D - F_{2n-2}$ has leading terms $F_{2n-1} + F_{2n-6} + F_{2n-8}$, which is not in $3\cA(G_n)$ by Corollary~\ref{Cor:3A}.
        
            \item  If $D$ has leading terms $F_{2n-3} + F_{2n-6} + F_{2n-9}$, then $D - 3F_{2n-4}$ has leading terms $F_{2n-1} + F_{2n-3} + F_{2n-9}$, which is not in $3\cA(G_n)$ by Corollary~\ref{Cor:3A}.

            \item If $D$ has leading terms $F_{2n-3}+F_{2n-6}+F_{2n-k}$ with $k\geq 10$. Then $D-5F_{2n-5}$ has leading terms $F_{2n-1}+F_{2n-4}+F_{2n-6}+F_{2n-8}+F_{2n-k}$. By the same argument as with $F_{2n-3}+F_{2n-6}$, $D-5F_{2n-5}\notin 3\cA(G_n).$
            \end{enumerate}

        \end{enumerate}

    \item  Let $D$ have leading term $F_{2n-2}$.

        \begin{enumerate}
        
        \item Let $D = F_{2n-2}$. Then $D - 5F_{2n-5} = F_{2n-1} + F_{2n-3} + F_{2n-6} + F_{2n-8}$, which is not in $3\cA(G_n) \setminus 2\cA(G_n)$ by Corollary~\ref{Cor:3A}. Moreover, by Corollary~\ref{Cor:2A}, $D-5F_{2n-5} \notin 2\cA(G_n)$. Altogether, $D-5F_{2n-5} \notin 3\cA(G_n)$.

        \item  Let $D$ have leading terms $F_{2n-2} + F_{2n-k}$ with $k \geq 7$.  By Corollary~\ref{Cor:3A}, this implies that $k \in \{ 7,8,9,10 \}$.  Then $D-8F_{2n-6}$ has leading terms $F_{2n-1} + F_{2n-3} + F_{2n-5}$.  If $D \neq F_{2n-2} + F_{2n-10}$, then there are additional terms.  By Corollary~\ref{Cor:3A}, the only element of $3\cA(G_n)$ with leading terms $F_{2n-1} + F_{2n-3} + F_{2n-5}$ is itself, so $D - 8F_{2n-6} \notin 3\cA(G_n)$.  If $D = F_{2n-2} + F_{2n-10}$, then $D - 3F_{2n-4} = F_{2n-1} + F_{2n-3} + F_{2n-5} + F_{2n-10}$.  Again, the only element of $3\cA(G_n)$ with leading terms $F_{2n-1} + F_{2n-3} + F_{2n-5}$ is itself, so $D - 3F_{2n-4} \notin 3\cA(G_n)$.

        \item Let $D$ have leading terms $F_{2n-2} + F_{2n-6}$.  Then by Corollaries~\ref{Cor:3A} and~\ref{Cor:2A}, we have $D \in \cA(G_n)$.  By Corollary~\ref{Cor:A}, either $D = F_{2n-2} + F_{2n-6}$, or the next term is $F_{2n-k}$ with $k \geq 9$.  We break this into further cases.
            
            \begin{enumerate}
            \item  If $D = F_{2n-2} + F_{2n-6}$, we have $D - 8F_{2n-6} = F_{2n-1} + F_{2n-3} + F_{2n-5} + F_{2n-7} + F_{2n-9}$, which is not in $3\cA(G_n)$ by Corollary~\ref{Cor:3A}.

            \item  If $k=9$, then $D - 8F_{2n-6}$ has leading term $F_{2n-0}$, hence $D-8F_{2n-6} \notin 3\cA(G_n)$.

            \item  If $k=10$, then either $D = F_{2n-2} + F_{2n-6} + F_{2n-10}$, or there are more terms.  If there are more terms, then the leading term of $D - 8F_{2n-6}$ is smaller than $F_{2n-6}$, so $D-8F_{2n-6} \notin 3\cA(G_n)$.  If $D = F_{2n-2} + F_{2n-6} + F_{2n-10}$, then $D - 2F_{2n-3} =  F_{2n-1} + F_{2n-3} + F_{2n-5} + F_{2n-8} + F_{2n-10}$, which is not in $3\cA(G_n)$ by Corollary~\ref{Cor:3A}.

            \item  If $k > 10$, then the leading terms of $D-8F_{2n-6}$ are $F_{2n-1} + F_{2n-3} + F_{2n-5} + F_{2n-7} + F_{2n-9} + F_{2n-k}$, so $D-8F_{2n-6} \notin 3\cA(G_n)$ by Corollary~\ref{Cor:3A}.
            
            \end{enumerate}

        \item Let $D$ have leading terms $F_{2n-2} + F_{2n-5} + F_{2n-k}$ with $k \geq 7$. Then $D -2F_{2n-3}$ has leading term $F_{2n-k}$. By Corollary~\ref{Cor:3A}, the smallest leading term possible for an element in $3\cA(G_n)$ is $F_{2n-6}$, so $D-2F_{2n-3} \notin 3\cA(G_n)$
    
        \item  Let $D = F_{2n-2} + F_{2n-5}$.  Then $D-8F_{2n-6} = F_{2n-8} + F_{2n-11}$.  Again, the smallest leading term possible for an element in $3\cA(G_n)$ is $F_{2n-6}$, so $D-2F_{2n-3} \notin 3\cA(G_n)$

        \item  Let $D$ have leading terms $F_{2n-2} + F_{2n-4}$.  By Corollary~\ref{Cor:3A}, if there is another term, then it is either $F_{2n-7}$ or smaller than $F_{2n-8}$. If $D$ has no other term, then $D - 5F_{2n-5}$ has leading terms $F_{2n-6} + F_{2n-8}$, which is not in $3\cA(G_n) \setminus \cA(G_n)$ by Corollary~\ref{Cor:3A}. Moreover, by Corollary~\ref{Cor:2A}, $D-5F_{2n-5} \notin \cA(G_n)$.
        
        If the next term is of the form $F_{2n-k}$ with $k\geq 7$, then $D-F_{2n-2}$ has leading terms $F_{2n-4} + F_{2n-k}$. Hence, $D-F_{2n-2} \notin 3\cA(G_n)$ by Corollary~\ref{Cor:3A}.

        \end{enumerate}

    \item  Let $D$ have leading term $F_{2n-1}$.

        \begin{enumerate}

        \item Let $D$ have leading terms $F_{2n-1} + F_{2n-6}$.  If there are no other terms, then $D - 8F_{2n-6} = F_{2n-4} + F_{2n-6} + F_{2n-8} + F_{2n-11}$, which is not in $3\cA(G_n)$ by Corollary~\ref{Cor:3A}.  Otherwise, the next term is either $F_{2n-9}$ or $F_{2n-10}$.  Then $D - 3F_{2n-4}$ has leading term $F_{2n-3}$, followed by either $F_{2n-9}$ or $F_{2n-10}$.  In either case, by Corollary~\ref{Cor:3A}, we see that $D - 3F_{2n-4} \notin 3\cA(G_n).$

        \item Let $D = F_{2n-1}$.  Then $D - 3F_{2n-4} = F_{2n-4} + F_{2n-7}$, which is not in $3\cA(G_n)$ by Corollary~\ref{Cor:3A}.

        \item Let $D = F_{2n-1} + F_{2n-5}$.  Then $D - 8F_{2n-6} = F_{2n-3} + F_{2n-8} + F_{2n-11}$, which is not in $3\cA(G_n)$ by Corollary~\ref{Cor:3A}.

        \item Let $D$ have leading terms $F_{2n-1} + F_{2n-3} + F_{2n-8}$.  Then $D - 2F_{2n-3}$ has leading terms $F_{2n-2} + F_{2n-8}$, which is not in $3\cA(G_n)$ by Corollary~\ref{Cor:3A}.

        \item Let $D$ have leading terms $F_{2n-1} + F_{2n-3} + F_{2n-7}$.  Then $D - 2F_{2n-3}$ has leading terms $F_{2n-2} + F_{2n-7}$. By Corollary~\ref{Cor:3A}, the only element in $3\cA(G_n)$ with these leading terms is $F_{2n-2} + F_{2n-7}$ itself. Hence, if there are more terms, then $D - 2F_{2n-3} \notin 3\cA(G_n)$.  Otherwise, if $D = F_{2n-1} + F_{2n-3} + F_{2n-7}$, then $D - 8F_{2n-6} = F_{2n-2} + F_{2n-6} + F_{2n-11}$, which is not in $3\cA(G_n)$ by Corollary~\ref{Cor:3A}.

        \item Let $D$ have leading terms $F_{2n-1} + F_{2n-3} + F_{2n-6}$. If $D = F_{2n-1} + F_{2n-3} + F_{2n-6}$, then $D - 8F_{2n-6} = F_{2n-2} + F_{2n-6} + F_{2n-8} + F_{2n-11}$, which is not in $3\cA(G_n)$ by Corollary~\ref{Cor:3A}.  Otherwise, the next term is either $F_{2n-9}$ or $F_{2n-10}$.  Then $D - 3F_{2n-4}$ has leading terms $F_{2n-2} + F_{2n-5}$, followed by either $F_{2n-9}$ or $F_{2n-10}$. By Corollary~\ref{Cor:3A}, this is not in $3\cA(G_n)$.

        \item  Let $D$ have leading terms $F_{2n-1} + F_{2n-3} + F_{2n-5}$.  Then $D = F_{2n-1} + F_{2n-3} + F_{2n-5}$ and $D - 8F_{2n-6} = F_{2n-2} + F_{2n-5} + F_{2n-8}+F_{2n-11}$. Again, by Corollary~\ref{Cor:3A}, this is not an element of $3\cA(G_n)$.

        \item  Finally, let $D$ have leading terms $F_{2n-1}+F_{2n-4}$.  If there is a next term, then it is $F_{2n-6}$.  We again consider further subcases.

            \begin{enumerate}
        
            \item  If $D$ has leading terms $F_{2n-1} + F_{2n-4} + F_{2n-6} + F_{2n-9}$, then $D - 8F_{2n-6}$ has leading term $F_{2n-2} + F_{2n-11}$.  Hence $D - 8F_{2n-6} \notin 3\cA(G_n).$
        
            \item  If $D$ has leading terms $F_{2n-1} + F_{2n-4} + F_{2n-6} + F_{2n-10}$, then $D - 8F_{2n-6}$ has leading term $F_{2n-2}$. If there is a next term, then $D - 8F_{2n-6}$ has leading terms $F_{2n-2} + F_{2n-k}$ with $k \geq 12$, which is not in $3\cA(G_n)$ by Corollary~\ref{Cor:3A}. \\

            The cases above reduce the problem to three specific divisors $D$.  Specifically, $D = F_{2n-1} + F_{2n-4}$, $D = F_{2n-1} + F_{2n-4} + F_{2n-6}$, and $D = F_{2n-1} + F_{2n-4} + F_{2n-6} + F_{2n-10}$.  We will use Dhar's burning algorithm starting at $v_8$ to show that the $v_8$-reduced divisors equivalent to these do not contain $v_8$ in their support.  As a consequence, none of these divisors have positive rank.  Equivalently, this shows that $D - 21F_{2n-8} \notin 3\cA(G_n)$ for these three divisors $D$.

            \item Let $D= F_{2n-1} + F_{2n-4}$.  Identifying $\Jac (G_n)$ with $\Z/F_{2n-2}\Z$ by the isomorphism $\varphi$, we have $D = 2v_0 + 2v_2$.  We shall run Dhar's burning algorithm starting at the vertex $v_8$.  Initially, everything burns except for $v_0$.  As such, we fire $v_0$ to obtain $v_1 + 3v_2$.  At the next step, everything burns except $v_0, v_1$ and $v_2$. As such, we fire these vertices to obtain $v_2 + 2v_3 + v_4$.  At this point, everything burns except for $v_0, v_1, v_2$ and $v_3$. Firing these vertices, we obtain $3v_4 + v_5$, at which point the entire graph burns.  It follows that $3v_4 + v_5$ is $v_8$-reduced.  Since it does not contain $v_8$ in its support, the divisor $D$ does not have positive rank.

            \item Let $D= F_{2n-1} + F_{2n-4} + F_{2n-6}$.  Equivalently, $D = 2v_0 + v_2 + v_4$.  We again run Dhar's burning algorithm starting at $v_8$.  Again, initially, everything burns except for $v_0$. As such, we fire $v_0$ to obtain $v_1 + 2v_2 + v_4$.  At the next step, everything burns except $v_0, v_1$ and $v_2$. Firing these vertices, we obtain  $2v_3 + 2v_4$.  Now, everything burns except for $v_0, v_1, v_2, v_3$ and $v_4$. Firing these vertices, we obtain $v_3 + 2v_5 + v_6$, at which point the entire graph burns.
       
            \item Let $D= F_{2n-1} + F_{2n-4} + F_{2n-6} + F_{2n-10}$.  Equivalently, $D = 2v_0 + v_2 + v_6$.  We once again run Dhar's burning algorithm starting at $v_8$.  Initially, everything burns except for $v_0$.  As such, we fire $v_0$ to obtain $v_1 + 2v_2 + v_6$.  At the next step, everything burns except $v_0, v_1$ and $v_2$. As such, we fire $v_0, v_1$ and $v_2$ to obtain $2v_3 + v_4 + v_6$, at which point the entire graph burns.
        \end{enumerate}

        \end{enumerate}

\end{enumerate}

\end{proof}

\begin{proof}[Proof of Theorem~\ref{Thm:StripGonality}]
By Theorem~\ref{Thm:MainThm}, $\gon (G_n) = \min \{ \lceil \frac{n}{2} \rceil , 5 \}$, and by Corollary~\ref{Cor:Gonality}, we have
\[
\gon(G_n) = \min \{ d \mid \exists x \in \Z/F_{2n}\Z \text{ such that } x-\cA(G_n) \subseteq (d-1)\cA(G_n) \} .
\]
Finally, by Corollary~\ref{Cor:SetA}, we have $x - \cA(G_n) \subseteq (d-1)\cA(G_n)$ if and only if $x$ satisfies the condition in the statement of the theorem.
\end{proof}

\nocite{*}
\bibliography{references}
\bibliographystyle{plain}

\end{document}